\documentclass[preprint,12pt]{article}
\usepackage{amssymb}
\usepackage{mathrsfs}
\usepackage{amsfonts}
\usepackage{graphicx}
\usepackage{indentfirst}
\usepackage{colortbl,dcolumn}
\usepackage{color}
\usepackage{amsmath}
\usepackage{psfrag}
\usepackage{booktabs}
\usepackage{cite}
\usepackage{amsthm}
\usepackage{float}
\usepackage{verbatim}
\usepackage{appendix}
\numberwithin{equation}
{section}
\usepackage[top=1in, bottom=1in, left=0.8in, right=0.8in]{geometry}

\newcommand{\R}{\mathbb{R}}
\newcommand{\N}{\mathbb{N}}
\newcommand{\E}{\mathbb{E}}

\newcommand{\dd}{\text{d}}
\newtheorem{thm}{Theorem}[section]
\newtheorem{defn}[thm]{Definition}
\newtheorem{lem}[thm]{Lemma}

\newtheorem{cor}[thm]{Corollary}

\newtheorem{assumption}[thm]{Assumption}

\begin{document}

\title{Order-one explicit 
approximations of random periodic solutions of semi-linear SDEs with multiplicative noise
\footnotemark[2] \footnotetext[2]{This work was supported by Natural Science Foundation of China (12471394, 12071488, 12371417). YW would like to acknowledge the support of the Royal Society through the International Exchanges scheme IES\textbackslash R3\textbackslash 233115. 
The authors want to thank Yinghao Huang for his suggestions and comments improving the manuscript.
                \\
                E-mail addresses:
 y.j.guo@foxmail.com, x.j.wang7@csu.edu.cn, 
yue.wu@strath.ac.uk. 
                }}
\author{Yujia Guo$^{a}$, Xiaojie Wang$^{a}$, and Yue Wu$^{b}$
\\
\footnotesize $^a$ School of Mathematics and Statistics, HNP-LAMA, Central South University, Changsha, \\
\footnotesize Hunan, P. R. China \\
\footnotesize $^b$ Department of Mathematics and Statistics, University of Strathclyde, Glasgow G1 1XH, UK} 
\date{\today}

\maketitle
\begin{abstract}
This paper is devoted to
order-one explicit approximations of random periodic solutions 
to multiplicative noise driven stochastic differential equations (SDEs) with non-globally Lipschitz coefficients. The existence of the random periodic solution is demonstrated as the limit of the pull-back of the discretized SDE. A novel approach is introduced to analyze mean-square error bounds of the proposed scheme that does not depend on a prior high-order moment bounds of the numerical approximations. Under mild assumptions, the proposed scheme is proved to achieve an expected order-one mean square convergence in the infinite time horizon.
Numerical examples are finally
provided to verify  the theoretical results.


\end{abstract}

\maketitle

\section{Introduction} Random dynamical systems are widely used across disciplines such as physics, biology, climatology, and finance to address uncertainties and random effects, playing a key role in studying natural phenomena. Random periodic solutions (RPS), as an example of mathematical tools recently developed in this field, are to study the long-time behavior of stochastic differential equations (SDEs) that also exhibit certain type of periodicity \cite{feng2011pathwise}.

The concept of random periodic solutions for a $C^1$-cocycle was first introduced by Zhao and Feng\cite{zhao2009random}, and later elaborated on by Feng, Zhao, and Zhou \cite{feng2011pathwise} for semiflows. Their research has driven further advancements in the study of a range of topics related to both autonomous and non-autonomous SDEs. This includes exploring the existence of random solutions admitted by non-autonomous SPDEs with additive noise \cite{feng2012random}, 
the anticipation of random solutions of SDEs with multiplicative linear noise \cite{feng2016anticipating}, periodic measures and ergodicity \cite{feng2020random}, among other areas.

Demonstrating random periodicity through numerical approximations is equally important, as random periodic solutions generally do not have a closed-form expression.
The relevant research on the numeric 
of SDEs has made rapid progress over last decades
(see \cite{hutzenthaler2020perturbation,
hutzenthaler2015numerical,
hutzenthaler2011strong,
kloeden1992stochastic,
milstein2004stochastic} and references therein).
It is worth noting that the numerical treatment for random periodic solutions is discussed over an infinite time horizon.
The initial study \cite{feng2017numerical} utilized traditional numerical techniques, such as the Euler-Maruyama method and a modified Milstein method, to estimate RPS in a dissipative system with global Lipschitz coefficients.
Moradi et al.\cite{moradi2024random} further this topic by simulating RPS using Euler and Milstein methods with weaker conditions on the drift term. 

Wu \cite{wu2023backward} investigated the existence and uniqueness of RPS of an additive SDE with a one-sided Lipschitz condition, and analyzed the order-half convergence of its numerical approximation using the backward Euler method.
Subsequently, Guo, Wang, and Wu
\cite{guo2023order} lifted the convergence order from half to one with a more relaxed condition  to \cite{wu2023backward}.
Chen, Cao, and Chen \cite{chen2024stochastic} studied stochastic theta methods for approximations of RPS and demonstrated that the mean square convergence order is half for SDEs with multiplicative noise and one for SDEs with additive noise under non-globally Lipschitz conditions.

Nevertheless, implicit methods are computationally expensive,
as it requires solving an implicit 
equation at each step \cite{andersson2017mean,higham2002strong,wang2020mean,wang2023mean}.
To save computational costs, many researchers 
turn to explicit time-stepping schemes for SDEs with non-globally Lipschitz coefficients (see, e.g., \cite{hutzenthaler2012strong,mao2015truncated,li2020explicit,wang2013tamed,sabanis2016euler,beyn2016stochastic,beyn2017stochastic,kelly2018adaptive}, to just mention a few).
Recently, some attempts have also been made to
numerically approximate random periodic solutions of SDEs under non-globally Lipschitz conditions.
For example,
an explicit time discretization scheme,
called
the projected Euler method 
(PEM), was introduced for 
this purpose \cite{guo2024projected}.
The mean square convergence rate of this approximation scheme has been proved to be 
order half for SDEs with multiplicative noise and order one for SDEs with additive noise.

In this paper, we further leverage the advantage of projected method under a weak condition (see Section \ref{sec_PMM:Random Periodic Solutions of SDEs})
and introduce an explicit Milstein schemes, termed projected Milstein method (PMM),
which is 
strongly convergent with order
one for SDEs with multiplicative noise over an infinite time horizon.

To be more precise, given $W:\mathbb{R} \times \Omega \rightarrow
\mathbb{R}^{m} $ a standard two-sided Wiener process on the probability space $(\Omega,\mathcal{F},\mathbb{P})$, where
the filtration is defined as
$\mathcal{F}^{t}_{s}:=
\sigma\{W_{u}-W_{v}:s\leq v\leq u\leq t\}$ and $\mathcal{F}^{t}
=\mathcal{F}^{t}_{-\infty}
=\bigvee_{s\leq t}\mathcal{F}^{t}_{s}$.
We consider the  following semi-linear stochastic differential equations (SDEs) with multiplicative noise:

\begin{equation}
\label{eq_PMM:Problem_SDE}
        \left\{
        \begin{aligned}
        \dd X_{t}^{t_{0}} 
         & = 
         \big(
         A X_{t}^{t_{0}}
         +
         f(t, \,X_{t}^{t_{0}})
         \big)
         \dd t
         + 
         g(t,
         \,X_{t}^{t_{0}})
         \, \dd W_t,
   \quad t \in (t_{0},T],\\
		 X_{t_{0}}^{t_{0}} & = \xi,
	\end{aligned}\right.
\end{equation}
where 
$A \in \R^{d \times d}$ is a negative-definite matrix,
$f:\mathbb{R} \times \mathbb{R}^{d}\rightarrow \mathbb{R}^{d}$,  $g:\mathbb{R} \times \mathbb{R}^{d} \rightarrow  \mathbb{R}^{d \times m}$ are continuous functions.
We denote the process $X_{t_1}^{t_0}$ evaluated at $t_1$, starting from $t_0$.
Additionally,
The random initial value $\xi$ is assumed to be $\mathcal{F}^{t_{0}}$-measurable. 
Note that by the variation of constant formula, the solution of \eqref{eq_PMM:Problem_SDE} can be written as
\begin{equation}
    \label{eq:variation of comstant formula}
    X^{t_{0}}_{t}(\xi)=
    e^{A (t-t_{0})}\xi
    +\int_{t_{0}}^{t}
    e^{A (t-s)} f(s,X^{t_{0}}_{s})
    \,\dd s
    +\int_{t_{0}}^{t}
    e^{A(t-s)} 
    g(s,X^{t_{0}}_{s})
    \,\dd W_{s}.
\end{equation}
Given a stepsize $h \in (0,1)$, we define the projection operator $\Phi:\R^d \rightarrow \R^d$ by
\begin{equation} 
\label{eq_PMM:projecion_funcaion}
    \left\{
    \begin{aligned}
    \Phi(x)
    &:=
    \min
    \Big\{
    1,
    h^{-\frac{1}{2\gamma}}
    \|x\|^{-1}
    \Big\}
    x,
    \quad x \neq 0,
    \\
    \Phi(x)
    &=0,
    \quad x = 0,
    \end{aligned}
    \right.
\end{equation}
where
$\gamma \geq 1$. 
The proposed projected Milstein method for 
SDEs \eqref{eq_PMM:Problem_SDE} starting at $-k\tau$ is given by
\begin{equation}
\label{eq_PMM:the_projected_milstein_method}
\left\{
\begin{aligned}
   &{\tilde{X}}
   ^{-k\tau}_{-k\tau+(j+1)h}
   =
   \Phi
   \big(
   {\tilde{X}}
   ^{-k\tau}_{-k\tau+jh}
   \big)
   +A h
   \Phi
   \big(
   {\tilde{X}}^{-k\tau}_{-k\tau+jh}
   \big)
   + hf
   \Big( 
   jh, 
   \Phi
   \big(
   {\tilde{X}}^{-k\tau}_{-k\tau+jh}
   \big)
   \Big) \\
   &\qquad\qquad\qquad+
   g \Big(
   jh,
   \Phi
   \big(
   {\tilde{X}}
   ^{-k\tau}
   _{-k\tau+jh}
   \big)
   \Big)
   \Delta W_{-k\tau+jh}
   +
   \sum_{r_1,r_2=1}^{m}
   \mathcal{L}^{r_1}
   g_{r_2}
   \Big(
   jh,
   \Phi
   \big(
   {\tilde{X}}^{-k\tau}_{-k\tau+jh}
   \big)
   \Big)
   \Pi
   ^{-k\tau+jh
   ,-k\tau+(j+1)h}
   _{r_1,r_2},\\
   &\tilde{X}^{-k\tau}_{-k\tau}
   =
   \xi,
\end{aligned}\right.
\end{equation}
for all $j \in \mathbb{N}$, 
where 
$\Delta W_{-k\tau+jh}:=
W_{-k\tau+(j+1)h}-W_{-k\tau+jh}$,
$\mathcal{L}^{r_1}g_{r_2}$ is a function from $ \R \times \R^d \rightarrow \R^d$ for $r_1,r_2 \in [m] ([m]:=
\{1,...,m\})$,
defined by
\begin{equation}
\label{eq_PMM:L_r}
    \mathcal{L}^{r_1}g_{r_2}(t,x)
    :=
    \sum_{k=1}^d
    g_{k,r_1}
    \dfrac{\partial g_{r_2}(t,x)}{\partial x^k}
    =
    \dfrac{\partial g_{r_2}(t,x)}{\partial x}
    g_{r_1}(t,x),
    \quad 
    t \in [0,\tau),
    x \in \R^d.\\
\end{equation}
\begin{equation}
    \Pi^{-k\tau+jh,
    -k\tau+(j+1)h}_{r_1,r_2}
    :=
    \int
    _{-k\tau+jh}
    ^{-k\tau+(j+1)h}
    \int_{-k\tau+jh}
    ^{s_2}
    \,\dd W_{s_1}^{r_1}
    \,\dd W_{s_2}^{r_2},
    \quad 
    r_1,r_2 \in [m].
\end{equation}
 The contribution of this article is two-fold:
\begin{itemize}
    \item In a non-globally Lipschitz setting, an explicit time-discretization scheme \eqref{eq_PMM:the_projected_milstein_method} of Milstein type is introduced, which admits a RPS converging to the RPS of SDEs \eqref{eq_PMM:Problem_SDE} with a  convergence rate of order one; 
    \item Without relying on a prior high-order moment bounds of the numerical approximations, the long-time mean-square convergence rate of the PMM \eqref{eq_PMM:the_projected_milstein_method} is obtained for SDEs whose drift and diffusion coefficients possibly grow superlinearly.
\end{itemize}
The paper is structured as follows. Section \ref{sec_PMM:Random Periodic Solutions of SDEs} discusses assumptions and the existence and uniqueness of RPS of SDEs. Section \ref{sec:PMM} presents the well-posedness and the existence of unique random periodic solutions using the explicit Milstein method. Mean square convergence for random periodic solutions of the explicit Milstein method is established in Section \ref{sec:strong_rate_of_PMM}. Section \ref{sec_PMM:numerical results} provides several numerical experiments to illustrate the theoretical results.

\section{Random Periodic Solutions of SDEs}
\label{sec_PMM:Random Periodic Solutions of SDEs}
In 2011, Feng, Zhao and Zhou
\cite{zhao2009random} gave
the definition of the random periodic solution for stochastic semi-flows.
Let $X$ be a separable Banach space.
Denote by
$ (\Omega,\mathcal{F},\mathbb{P},(\theta _{s})_{s \in \mathbb{R}}
)$ 
a metric dynamical system and 
$\theta _{s}:\Omega \rightarrow \Omega$ 
is assumed to be a measurably invertible 
for all $s \in \mathbb{R}$.
Denote $\Delta:= \{(t,s)  \in \mathbb{R}^{2},s \leq t \}$.  
Consider a stochastic  periodic semi-flow 
$u:\Delta \times \Omega \times X \rightarrow X $ of period $\tau$, 
which satisfies the following standard condition
\begin{equation}
    u(t,r,\omega) =
    u(t,s,\omega) 
    \circ
    u(s,r,\omega),
\end{equation}
and the periodic property
\begin{equation}
u(t+\tau,s+\tau,\omega)=
u(t,s,\theta_{\tau}\omega),
\end{equation}
for all $r \leq s \leq t,r,s,\in \mathbb{R}$, 
for a.e. $\omega \in \Omega$.

\begin{defn}
\label{def:rps}
	A random periodic solution of period $\tau
       >0 $ of a semi-flow
	 $u:\Delta \times \Omega \times X \rightarrow X 
       $
	  is an $\mathcal{F}$-measurable map 
	  $Y: \mathbb{R} \times \Omega \rightarrow X $ 
	  such that
	\begin{equation}
		u(t+\tau,t,Y(t,\omega),\omega)=
		Y(t+\tau,\omega)=
		Y(t,\theta_{\tau}\omega),
	\end{equation}
	for any $(t,s) \in \Delta$, $\omega \in \Omega$.
\end{defn}

Throughout this paper the following notation
is frequently used.
For simplicity, we denote $[d]:=
\{1,...,d\}$ and the letter
$ C $ is used to denote a generic positive
constant independent of time step size and
may vary for each appearance. 
Let $|\cdot|$, $ \| \cdot \| $ and 
$ \left\langle \cdot , \cdot \right\rangle $
be the absolute value of a scalar, the Euclidean norm and the inner product of vectors
in $ \R^d $, respectively.
By $A^{T}$ we denote the transpose of vector or matrix.
Given a matrix $A$, we use $\| A \|:=\sqrt{trace(A^{T}A)}$ to denote the trace norm of $A$.
On a probability space 
$ ( \Omega, \mathcal{ F }, \mathbb{P} ) $, 
we use $ \E $ to denote the mean expectation and
$ L^p(\Omega; \textcolor{black}{\R^{d }}) $, $ d \in \mathbb{N} $,
to denote the family of 
$ \textcolor{black}{\R^{d}} $-valued variables with
the norm defined by 
$ \| \xi \|_{L^p(\Omega;\textcolor{black}{\R^{d}})} 
=(\E [\|\xi \|^p])^{\frac{1}{p}} < \infty $.

We consider the following assumptions require to establish our main result.
\begin{assumption}
\label{ass_PMM}
Suppose the following conditions are satisfied.

(\romannumeral1)
$A$ is self-adjoint and negative definite and
there exists a non-decreasing sequence $(\lambda_{i})_{i \in [d] } \subset \mathbb{R}$ 
of positive real numbers 
and an orthonormal basis $(e_{i})_{i \in [d]}$, 
such that 
$A e_{i} =-\lambda_{i} e_{i}$,  $i \in [d]$.

(\romannumeral2)
The drift coefficient functions $f$ and diffusion coefficient functions $g$ are continuous and periodic in time with period $\tau >0$, i.e.,
\begin{equation}
    f(t+\tau,x)=f(t,x),
    \quad
   g(t+\tau,x)=g(t,x),
    \qquad
    \forall
    x \in \R^{d},
    t \in \R.
\end{equation}

(\romannumeral 3)
For some $p^{*}\in [1,\infty)$,
there exists a constant $K_1< \lambda_1$ such that
\begin{equation}
\label{eq_PMM:coercivity_condition}
     \langle 
    x, 
    f(t,x)
    \rangle
    +
    \dfrac{2p^{*}-1}{2}
    \|
    g(t,x)
    \|^{2}
     \leq 
    K_1
    (
    1
    +
    \|x\|^{2}
    ),
    \qquad 
    \forall x \in \R^d,
    t \in [0,\tau).
\end{equation}

(\romannumeral 4)
For any $p \in [1,p^*)$, 
there exists a constant $C^{*}>0$ depending on $p$ such that 
$\E \big[\|\xi\|^{2p} \big]\leq C^{*}$.
\end{assumption}
With Assumption \ref{ass_PMM},
one obtains
\begin{equation}
\label{eq:lambda}
    \langle
    x,Ax
    \rangle
    \leq
    - \lambda_{1}
    \| x \|^{2},
    \qquad 
    \forall x \in \R^{d}.
\end{equation}
The following assumption ensures the existence and uniqueness of the random periodic solution of \eqref{eq_PMM:Problem_SDE}
under non-globally Lipschitz conditions.

\begin{assumption}
\label{thm:unique_random_periodic_solution}
Assume that
there exists a unique random periodic solution $X_{t}^{*}(\cdot) \in L^{2}(\Omega)$  in the form \begin{equation}
    \label{eq:pull-back}
    X^{*}_{t}
    =
    \int_{-\infty}^{t}
    e^{A(t-s)}
    f(s,X^{*}_{s})
    \,\dd s
    +\int_{-\infty}^{t}
    e^{A(t-s)}
    g(s,X^{*}_{s})
    \,\dd W_{s},
\end{equation}
such that $X^{*}$ is a limit of the pull-back 
$X^{-k\tau}_{t}(\xi)$ of \eqref{eq_PMM:Problem_SDE} when $k \rightarrow \infty$, i.e., 
	\begin{equation}
		\lim_{k \rightarrow \infty }
        \E
        \Big[
        \big\| X_{t}^{-k\tau}(\xi) -X_{t}^{*} \big\|^{2}
        \Big]
        =0.
	\end{equation}
\end{assumption}

The boundedness of the exact 
of SDE
\eqref{eq_PMM:Problem_SDE} 
has been established in \cite[Lemma 2.6]{guo2024projected} as follows.
\begin{lem}
\label{lem_PMM:the_pth_of_exact_solution}
    Let Assumption \ref{ass_PMM} hold and let $X_{t}^{-k\tau}$ be the exact solution of the SDE \eqref{eq_PMM:Problem_SDE}. If the initial value $X^{-k\tau}_{-k\tau}=\xi$,
    then for any $p \in [1,2p^*)$, there exists a positive constant $C$ 
    such that 
    \begin{equation}\label{eq_PMM:the_p_th_(X(t))}
     \sup_{t\geq -k\tau}\E
     \Big[ 
     {\big\| X^{-k\tau}_{t} \big\|}^{2p} 
     \Big]
     \leq
      C
      \big(
      1
      +
      \E[\|\xi\|^{2p}]
      \big)
      <
      \infty.
\end{equation}
\end{lem}

Before proceeding, we present some useful lemma, which has been given in \cite[Lemma 5.7]{pang2023projected}.

\begin{lem}
\label{lem_PMM:x-Phi(x)}
    Let $\gamma \geq 1$ be
    given by Assumption
    \ref{ass_PMM：polynomial_growth}.
    For all $ x \in \R^d$, let $\Phi(x)$
    be defined as \eqref{eq_PMM:projecion_funcaion}.
    Then for the case 
    $\gamma >1$,
    one has
    \begin{equation}
        \|
        x
        -
        \Phi(x)
        \|
        \leq
        Ch^2
        \|
        x
        \|
        ^{4\gamma+1},
        \quad
        x \in \R^d.
    \end{equation}
    In addition,
    when $\gamma=1$,
    one has $x-\Phi(x)=0$
    for all $ x \in \R^d$.
\end{lem}

Following a similar argument as  
\cite[Lemma 2.7]{guo2024projected}, 
we can readily establish the following bounds.

\begin{lem}
\label{lem_PMM:the_esti_(X(t1)-X(t2)}
    Let
    Assumption \ref{ass_PMM}  
    hold
    and let $X_{t}^{-k\tau}$ be the exact solution to the SDE
\eqref{eq_PMM:Problem_SDE}.
    Then there exists a positive constant $C$ which depends on $\gamma,d,A,f,g$ only,
    for all $t_1,t_2 \geq -k\tau$ and
    $p \in 
    \Big[2,\frac{2p^*}{\gamma}\Big)$,
    such that
    \begin{equation}
\label{eq_PMM:the_esti_(X(t1)-X(t2)}
    \begin{split}
        \big\|
        X_{t_1}^{-k\tau}
        -
        X_{t_2}^{-k\tau}
        \big\|
        _{L^p(\Omega;\mathbb{R}^d)}
        & \leq
        C
        \Big(
          1+\sup_{k\in \mathbb{N}} \sup_{t \geq -k\tau}
          \big\| 
          X_{t}^{-k\tau} 
         \big\|
         ^\gamma_{L^{p\gamma}(\Omega;\mathbb{\R}^d)}
        \Big)
        |t_{2}-t_{1}| \\
        &\quad+
        C
       \Big(
          1+\sup_{k\in \mathbb{N}} \sup_{t\geq -k\tau}
          \big\| 
          X_{t}^{-k\tau} 
         \big\|
         ^{\gamma}
         _{L^{p\gamma}(\Omega;\mathbb{R}^d)}
        \Big)
        |t_{2}-t_{1}|^{\frac{1}{2}}.
    \end{split}
\end{equation}
\end{lem}

The next lemma establishes the following bound of  time continuity in projection environment.
\begin{lem}
\label{lem_PMM:the_esti_(X(t1)-PhiX(t2)}
    Let
    Assumption \ref{ass_PMM}  
    hold and let
    $X_{t}^{-k\tau}$ be the exact solution to the SDE
\eqref{eq_PMM:Problem_SDE}.
    Then there exists a positive constant $C$ which depends on $\gamma,d,A,f,g$ only,
    for all $|t_1-t_2|\leq h$ and
    $p \in 
    \Big[1,\frac{2p^*}{4\gamma+1}\Big)$,
    such that
    \begin{equation}
    \label{eq_PMM:the_esti_X(t1)-PhiX(t2)}
        \big\|
        X_{t_1}^{-k\tau}
        -
        \Phi(X_{t_2}^{-k\tau})
        \big\|
        _{L^p(\Omega;\mathbb{R}^d)}
         \leq
        C h^{\frac{1}{2}}
        \Big(
          1+\sup_{k\in \mathbb{N}} \sup_{t \geq -k\tau}
          \big\| 
          X_{t}^{-k\tau} 
         \big\|
         ^{4\gamma+1}
         _{L^{p(4\gamma+1)}(\Omega;\mathbb{\R}^d)}
        \Big).
\end{equation}
\end{lem}
\begin{proof}
[Proof of Lemma \ref{lem_PMM:the_esti_(X(t1)-PhiX(t2)}]
A triangle inequality yields
        \begin{equation}
            \begin{split}
                \big\|
                X_{t_1}^{-k\tau}
                -
                \Phi(X_{t_2}^{-k\tau})
                \big\|
                _{L^p(\Omega;
                \mathbb{R}^d)}
                 \leq
                \big\|
                X_{t_1}^{-k\tau}
                -
                X_{t_2}^{-k\tau}
                \big\|
                _{L^p(\Omega;
                \mathbb{R}^d)}
                +
                \big\|
                X_{t_2}^{-k\tau}
                -
                \Phi(X_{t_2}^{-k\tau})
                \big\|
                _{L^p(\Omega;
                \mathbb{R}^d)}.
            \end{split}
        \end{equation}
        Based on Lemma \ref{lem_PMM:the_esti_(X(t1)-X(t2)},
        for $|t_1-t_2| \leq h \in (0,1)$,
        one can get,
        \begin{equation}
        \begin{split}
        \big\|
        &
        X_{t_1}^{-k\tau}
        -
        X_{t_2}^{-k\tau}
        \big\|
        _{L^p(\Omega;\mathbb{R}^d)} 
        \\
        & \leq
        C
        \Big(
        1+\sup_{k\in \mathbb{N}} \sup_{t \geq -k\tau}
        \big\| 
        X_{t}^{-k\tau} 
         \big\|
         ^\gamma_{L^{p\gamma}(\Omega;\mathbb{\R}^d)}
        \Big)
        |t_{2}-t_{1}| 
        +
        C
       \Big(
          1+\sup_{k\in \mathbb{N}} \sup_{t\geq -k\tau}
          \big\| 
          X_{t}^{-k\tau} 
         \big\|
         ^{\gamma}
         _{L^{p\gamma}(\Omega;\mathbb{R}^d)}
        \Big)
        |t_{2}-t_{1}|^{\frac{1}{2}}
        \\
        & \leq
        Ch^{\frac{1}{2}}
        \Big(
          1+\sup_{k\in \mathbb{N}} \sup_{t\geq -k\tau}
          \big\| 
          X_{t}^{-k\tau} 
         \big\|
         ^{\gamma}
         _{L^{p\gamma}(\Omega;\mathbb{R}^d)}
        \Big).
    \end{split}
    \end{equation}
    For the second term,
    applying Lemma \ref{lem_PMM:x-Phi(x)} leads to
    \begin{equation}
        \begin{split}
        \big\|
        X_{t_2}^{-k\tau}
        -
        \Phi(X_{t_2}^{-k\tau})
        \big\|
        _{L^p(\Omega;
        \mathbb{R}^d)}
        & \leq
        Ch^2
        \|
        X_{t_2}^{-k\tau}
        \|
        ^{4\gamma+1}
        _{L^{p(4\gamma+1)}} \\
        & \leq
        Ch^2
        \Big(
          1+\sup_{k\in \mathbb{N}} \sup_{t\geq -k\tau}
          \big\| 
          X_{t}^{-k\tau} 
         \big\|
         ^{4\gamma+1}
         _{L^{p(4\gamma+1)}(\Omega;\mathbb{R}^d)}
        \Big).
        \end{split}
    \end{equation}
    Therefore,
    \begin{equation}
        \big\|
        X_{t_1}^{-k\tau}
        -
        \Phi(X_{t_2}^{-k\tau})
        \big\|
        _{L^p(\Omega;\mathbb{R}^d)}
         \leq
        C h^{\frac{1}{2}}
        \Big(
          1+\sup_{k\in \mathbb{N}} \sup_{t \geq -k\tau}
          \big\| 
          X_{t}^{-k\tau} 
         \big\|
         ^{4\gamma+1}
         _{L^{p(4\gamma+1)}(\Omega;\mathbb{\R}^d)}
        \Big).
    \end{equation}
    This completes the proof.
\end{proof}

\section{Numerical Approximation of Random Periodic solutions}
\label{sec:PMM}
In this paper, we propose an explicit Milstein type
method to approximate the exact solution of the 
SDEs \eqref{eq_PMM:Problem_SDE}
starting at $-k\tau$,
\begin{equation}
\label{eq:the_projected_milstein_method}
\begin{split}
   {\tilde{X}}
   ^{-k\tau}_{-k\tau+(j+1)h}
   &=
   \Phi
   \big(
   {\tilde{X}}
   ^{-k\tau}_{-k\tau+jh}
   \big)
   +A h
   \Phi
   \big(
   {\tilde{X}}^{-k\tau}_{-k\tau+jh}
   \big)
   + hf
   \Big( 
   jh, 
   \Phi
   \big(
   {\tilde{X}}^{-k\tau}_{-k\tau+jh}
   \big)
   \Big) \\
   &\quad+
   g \Big(
   jh,
   \Phi
   \big(
   {\tilde{X}}^{-k\tau}_{-k\tau+jh}
   \big)
   \Big)
   \Delta W_{-k\tau+jh}
   +
   \sum_{r_1,r_2=1}^{m}
   \mathcal{L}^{r_1}
   g_{r_2}
   \Big(
   jh,
   \Phi
   \big(
   {\tilde{X}}^{-k\tau}_{-k\tau+jh}
   \big)
   \Big)
   \Pi
   ^{-k\tau+jh,
   -k\tau+(j+1)h}
   _{r_1,r_2},
\end{split}
\end{equation}
where $\Phi$ is defined by \eqref{eq_PMM:projecion_funcaion}.
Because of the periodicity of $f$ and $g$, we have that 
$f(-k\tau+jh,{\tilde{X}}^{-k\tau}_{-k\tau+jh}) =f(jh,{\tilde{X}}^{-k\tau}_{-k\tau+jh})$, 
$g(-k\tau+jh,{\tilde{X}}^{-k\tau}_{-k\tau+jh}) =g(jh,{\tilde{X}}^{-k\tau}_{-k\tau+jh})$.
We use $t_j$ to denote the time $-k\tau+jh$.
and use the notation $\Delta W_{r,-k\tau+jh}:=W_{r,-k\tau+(j+1)h}-W_{r,-k\tau+jh}, r \in [m]$.

In many applications,
the considered SDE systems possess commutative noise \cite{kloeden1992stochastic,
milstein2004stochastic}, namely,
the diffusion 
$g$ fulfills the so-called commutativity condition:
\begin{equation}
    \label{eq_PEM:Lg}
    \mathcal{L}^{r_1}
    g_{k,{r_2}}
    =
    \mathcal{L}^{r_2}
    g_{k,{r_1}},
    \quad 
    r_1,r_2 \in [m],
    k \in [d].
\end{equation}
Thanks to the property
$
\Pi
^
{
-k\tau+jh,
-k\tau+(j+1)h
}_{r_1,r_2}
+
\Pi
^
{
-k\tau+jh,
-k\tau+(j+1)h
}_{r_2,r_1}
=
\Delta
W_{r_1,-k\tau+jh}
\Delta
W_{r_2,-k\tau+jh}
,
r_1 \neq r_2
$,
in this case the explicit Milstein method \eqref{eq:the_projected_milstein_method}
takes a simple form as
\begin{equation}
\label{eq_PMM:PMM method} 
\begin{split}
   &{\tilde{X}}
   ^{-k\tau}
   _{-k\tau+(j+1)h}
   \\
   &=
   \Phi
   \big(
   {\tilde{X}}
   ^{-k\tau}_{-k\tau+jh}
   \big)
   +A h
   \Phi
   \big(
   {\tilde{X}}^{-k\tau}_{-k\tau+jh}
   \big)
   + hf
   \Big( 
   jh, 
   \Phi
   \big(
   {\tilde{X}}^{-k\tau}_{-k\tau+jh}
   \big)
   \Big) 
   +
   g \Big(
   jh,
   \Phi
   \big(
   {\tilde{X}}^{-k\tau}_{-k\tau+jh}
   \big)
   \Big)
   \Delta W_{-k\tau+jh} \\
   & \quad +
   \tfrac{1}{2}
   \sum_{r_1,r_2=1}^{m}
   \mathcal{L}^{r_1}
   g_{r_2}
   \Big(
   jh,
   \Phi
   \big(
   {\tilde{X}}^{-k\tau}_{-k\tau+jh}
   \big)
   \Big)
    (\Delta W_{r_1,-k\tau+jh}
    \Delta W_{r_2,-k\tau+jh}
    -\pi_{r_1,r_2}
    h),
\end{split}
\end{equation}
with
\begin{equation}
\label{eq_PMM:the_def_pi_r}
    \pi_{r_1,r_2}
    =
    \left\{
    \begin{aligned}
        1, \quad r_1&=r_2\\
        0, \quad r_1&\neq r_2
    \end{aligned}
    \right.
    ,
    \quad
    r_1,r_2 \in [m].
\end{equation}

We set up a general framework by making two key assumptions as follows.
\begin{assumption}
\label{ass_PMM:generalizedz_monotonicity_condition}
Assume that the diffusion coefficients 
$g_r:\R \times \R^d \rightarrow 
\R^{d }, r \in [m]$
are differentiable, 
and there exist constants $q \in [1,\infty)$ and $K_2 \in [0,\infty)$
such that, $\forall x,y \in \R^d,
t \in [0,\tau), h \in (0,1)$,
the drift and diffusion coefficients of SDE \eqref{eq_PMM:Problem_SDE} obey
\begin{equation}
\label{eq_PMM:generalizedz_monotonicity_condition}
    \langle 
    x-y , 
    f(t,x)-f(t,y)
    \rangle
    +
    \dfrac{2q-1}{2}
    \|
    g(t,x)-g(t,y)
    \|^{2}
    \leq 
    K_2 
     \| x-y \|^{2}.
\end{equation}
\end{assumption}

\begin{assumption}
\label{ass_PMM：polynomial_growth}
Assume the drift coefficient function $f$ of the SDE
\eqref{eq_PMM:Problem_SDE} is continuously differentiable and the diffusion coefficient function $g$ is  twice
continuously differentiable.
Moreover, there exist some positive constants
$\gamma \geq 1$  and
$ p^* > 5\gamma$,
such that,
$x,\bar{x},y \in \R^d$ and $t,s \in [0,\tau)$
\begin{align}
    \label{eq_PMM:the_esti_f'(t,x)-f'(t,y)}
    \big\|
    \big(
    \tfrac{\partial f}{\partial x}(t,x)
    -
    \tfrac{\partial f}{\partial x}(t,\bar{x})
    \big)
    y
    \big\|
    &\leq
    C(1+\|x\|+\|\bar{x}\|)
    ^{\max\{\gamma-2,0\}}
    \|x-\bar{x}\|
    \cdot
    \|y\|,\\
\label{eq_PMM:the_esti_g'(t,x)-g'(t,y)}
    \big\|
    \big(
    \tfrac{\partial g_r}{\partial x}(t,x)
    -
    \tfrac{\partial g_r}{\partial x}(t,\bar{x})
    \big)
    y
    \big\|
    &\leq
    C(1+\|x\|+\|\bar{x}\|)
    ^{\max\big\{\frac{\gamma-3}{2},0\big\}}
    \|x-\bar{x}\|
    \cdot
    \|y\|,\\
\label{eq_PMM:the_esti_f(t,x)-f(s,x)}
    \|
    f(t,x)-
    f(s,x)
    \|
    &\leq
    C(1+\|x\|)^\gamma|t-s|,\\
    \label{eq_PMM:the_esti_g(t,x)-g(s,x)}
    \|
    g_r(t,x)-
    g_r(s,x)
    \|
    &\leq
    C(1+\|x\|)^{\frac{\gamma+1}{2}}|t-s|.
\end{align}
In addition,
assume the vector functions $\mathcal{L}^{r_1}g_{r_2}:
\R \times \R^d \rightarrow \R^d$ are continuously differentiable and
\begin{equation}
    \label{eq_PMM:the_esti_L'(t,x)-L'(t,y)}
    \Big\|
    \Big(
    \tfrac{\partial\mathcal{L}^{r_1}g_{r_2}}{\partial x}(t,x)
    -
    \tfrac{\partial\mathcal{L}^{r_1}g_{r_2}}{\partial x}(t,\bar{x})
    \Big)
    y
    \Big\|
    \leq
    C(1+\|x\|+\|\bar{x}\|)
    ^{\max\{\gamma-2,0\}}
    \cdot
    \|y\|,
    \quad
    \forall x,\bar{x},y \in \R^d.
\end{equation}
\end{assumption}
Assumption \ref{ass_PMM：polynomial_growth} is considered as a kind of polynomial growth condition and in proofs which follow we will need some implications of this assumption.
It follows immediately 
from 
\eqref{eq_PMM:the_esti_f'(t,x)-f'(t,y)}
-
\eqref{eq_PMM:the_esti_g(t,x)-g(s,x)}
that $\forall x,y \in \R^d, t \in [0,\tau)$, 
\begin{align}
    \label{eq_PMM:the_esti_f'(t,x)}
    \big\|
    \tfrac{\partial f}{\partial x}(t,x)
    y
    \big\|
    &\leq
    C(1+\|x\|)
    ^{\gamma-1}
    \|y\|,\\
    \label{eq_PMM:the_esti_g'(t,x)}
    \big\|
    \tfrac{\partial g_r}{\partial x}(t,x)
    y
    \big\|
    &\leq
    C(1+\|x\|)
    ^{\tfrac{\gamma-1}{2}}
    \|y\|,
\end{align}
which in turn gives,
$\forall 
x,\bar{x}\in \R^d, t \in [0,\tau)$
\begin{align}
    \label{eq_PMM:the_esti_f(t,x)-f(t,y)}
    \|
    f(t,x)-f(t,\bar{x})
    \|
    &\leq
    C_1(1+\|x\|+\|\bar{x}\|)^{\gamma-1}
    \|x-\bar{x}\|,\\
    \label{eq_PMM:the_esti_f(t,x)}
    \|
    f(t,x)
    \|
    &\leq
    C_2(1+\|x\|)^{\gamma},\\
    \label{eq_PMM:the_esti_g(t,x)-g(t,y)}
    \|
    g_r(t,x)-g_r(t,\bar{x})
    \|
    &\leq
    C(1+\|x\|+\|\bar{x}\|)
    ^{\frac{\gamma-1}{2}}
    \|x-\bar{x}\|,\\
    \label{eq_PMM:the_esti_g(t,x)}
    \|
    g_r(t,x)
    \|
    &\leq
    C(1+\|x\|)^{\frac{\gamma+1}{2}}.
\end{align}
Similarly, \eqref{eq_PMM:the_esti_L'(t,x)-L'(t,y)} in Assumption \ref{ass_PMM：polynomial_growth} yields,
$\forall 
x,\bar{x}\in \R^d, t \in [0,\tau)$
\begin{align}
    \label{eq_PMM:the_esti_L(t,x)-L(t,y)}
    \|
    \mathcal{L}^{r_1}g_{r_2}(t,x)
    -
    \mathcal{L}^{r_1}g_{r_2}(t,\bar{x})
    \|
    &
    \leq
    C_3(1+\|x\|+\|\bar{x}\|)^{\gamma-1}
    \|x-\bar{x}\|,\\
    \label{eq_PMM:the_esti_L(t,x)}
    \|
    \mathcal{L}^{r_1}g_{r_2}(t,x)
    \|
    &\leq
    C_4(1+\|x\|)^{\gamma}.
\end{align}

Before proceeding further, 
we collect some preliminary estimates,
which have been established in
\cite[Lemma 4.2]{pang2024linear}.
\begin{lem}
\label{lem_PMM:Phi(x)}
Let Assumption \ref{ass_PMM}
hold.
Then the following estimates
hold
\begin{align}
    \label{eq_PMM:the_esti_Phi(x)}
    \|
    \Phi (x)
    \|
    &\leq
    h^{-\frac{1}{2\gamma}}, \\
    \label{eq_PMM:the_esti_f(t,Phi(x))}
    \|
    f(t,\Phi (x))
    \|
    &\leq
    M_fh^{-\frac{1}{2}}, \\
    \label{eq_PMM:the_esti_L(t,Phi(x))}
    \|
    \mathcal{L}^{r_1}g_{r_2}(t,\Phi (x))
    \|
    &\leq
    M_{\mathcal{L}}h^{-\frac{1}{2}}, 
\end{align}
for any 
$x \in \mathbb{R}^{d}$ and $t \in [0,\tau)$,
where
$M_f:=2C_2$, 
$M_{\mathcal{L}}:=2C_4$
 with
$C_2,C_4$  from \eqref{eq_PMM:the_esti_f(t,x)}
and
\eqref{eq_PMM:the_esti_L(t,x)},
respectively.
Moreover, for any
$x,y \in \mathbb{R}^{d}$
 and $t \in [0,\tau)$,
the following estimates hold true
\begin{align}
    \label{eq:the_esti_Phi(x)-Phi(y)}
    \|
    \Phi(x)-
    \Phi(y)
    \|
    & \leq
    \| x-y \|,
    \\
    \label{eq:the_esti_f(t,Phi(x))-f(t,Phi(y))}
    \|
    f(t,\Phi(x))
    -
    f(t,\Phi(y))
    \|
    & \leq
    \beta_f 
    h
    ^{-\frac{\gamma-1}{2\gamma}}
    \|x-y\|, \\
    \label{eq:the_esti_L(t,Phi(x))-L(t,Phi(y))}
    \|
    \mathcal{L}
    ^{r_1}g_{r_2}(t,\Phi (x))
    -
    \mathcal{L}
    ^{r_1}g_{r_2}(t,\Phi (x))
    \|
    & \leq
    \beta_{\mathcal{L}}
    h
    ^{-\frac{\gamma-1}{2\gamma}}
    \|x-y\|,
\end{align}
where 
$\beta_f:=3C_1$, 
$\beta_{\mathcal{L}}:=3C_3$,
and $C_1$ is from \eqref{eq_PMM:the_esti_f(t,x)-f(t,y)},
$C_3$ is from \eqref{eq_PMM:the_esti_L(t,x)-L(t,y)}.
Especially, 
taking $y=0$ in \eqref{eq:the_esti_Phi(x)-Phi(y)}, we have
for $x \in \R^d$
\begin{equation}
    \label{eq_PMM:Phi(x)}
    \|
    \Phi(x)
    \|
    \leq
    \|x\|.
\end{equation}
\end{lem}

\subsection{The local one-step approximation error}
\label{sec:strong_rate_of_PMM}
We focus on the analysis of the mean-square 
rate of the numerical scheme \eqref{eq:the_projected_milstein_method}.
The exact solution at time $-k\tau+(j+1)h$ can be decomposed as follows:
\begin{equation}
    \label{eq_PMM:the_exact_X_j+1}
    \begin{split}
        X_{-k\tau+(j+1)h}^{-k\tau} 
        &=
        \Phi(X_{-k\tau+jh}^{-k\tau})
        +Ah
        \Phi(X_{-k\tau+jh}^{-k\tau})
        +hf
        \big(jh,
        \Phi(X_{-k\tau+jh}^{-k\tau})
        \big) \\
        &\quad +
        g\big(jh,
        \Phi(X_{-k\tau+jh}^{-k\tau})
        \big)
        \Delta W_{-k\tau+jh}
        +
        \sum_{r_1,r_2=1}^{m}
        \mathcal{L}^{r_1}g_{r_2}
        \big(jh,
        \Phi(X_{-k\tau+jh}^{-k\tau})
        \big)
        \Pi
        ^{-k\tau+jh,
        -k\tau+(j+1)h}
        _{r_1,r_2}
        \\
        & \quad +
        \mathcal{R}_{-k\tau+(j+1)h},
    \end{split}
\end{equation}
where we denote
\begin{equation}
\label{eq_PMM:the_def_R_j+1}
    \begin{split}
        &\mathcal{R}_{-k\tau+(j+1)h} \\
        & :=
        \int_{-k\tau+jh}^{-k\tau+(j+1)h}
        A\big(
          X_{s}^{-k\tau}
          -
          \Phi
          (X_{-k\tau+jh}^{-k\tau})
        \big)
        \, \dd s +
        \int_{-k\tau+jh}^{-k\tau+(j+1)h}
          f
        \big(
          s,
         X_{s}^{-k\tau}
          \big)
          -
          f
         \big(
          jh,
         \Phi
         (X_{-k\tau+jh}^{-k\tau})
          \big)
        \, \dd s \\
        &\quad +
        \int_{-k\tau+jh}^{-k\tau+(j+1)h}
          g
         \big(
          s,
         X_{s}^{-k\tau}
          \big)
          -
          g
         \big(
          jh,
         \Phi
         (X_{-k\tau+jh}^{-k\tau})
          \big)
        \, \dd W_{s}
        \\
         & \quad \quad
         - 
          \sum_{r_1,r_2=1}^{m}
        \mathcal{L}^{r_1}g_{r_2}
        \big(jh,
        \Phi(X_{-k\tau+jh}^{-k\tau})
        \big)
        \Pi^{-k\tau+jh,
        -k\tau+(j+1)h}
        _{r_1,r_2}
        +
        X_{-k\tau+jh}^{-k\tau}
        -
        \Phi
        (X_{-k\tau+jh}^{-k\tau}).
    \end{split}
\end{equation}

The subsequent lemma provides uniform bounded estimates for the second moment of 
$\mathcal{R}_{-k\tau+(j+1)h}$
and its conditional expectation 
$ \E[\mathcal{R}_{-k\tau(j+1)h}| \mathcal{F}_{-k\tau+jh}]$.
\begin{lem}
\label{lem_PMM:the_esti_R_j+1}
    Let 
    Assumptions \ref{ass_PMM},
    \ref{ass_PMM:generalizedz_monotonicity_condition},
    \ref{ass_PMM：polynomial_growth}
     hold.
    Then for $k,j \in \N$,
    there exists some positive constant $C$,
    independent of $k,j$ and $h$,
    $\gamma \in \big[1,\frac{p^*}{5}\big]$,
    such that
    \begin{equation}
    \begin{split}
        \|
        \mathcal{R}_{-k\tau+(j+1)h}
        \|_{L^2(\Omega;\R^d)}
        &\leq
        Ch^{\frac{3}{2}}
         \Big(
        1+\sup_{k \in \N}
        \sup_{t\geq -k\tau}
        \|X_{t}^{-k\tau}\|
        ^{
        \frac{\max\{16\gamma+4,
        17\gamma+1
        \}}{2}}_{L^{\max\{16\gamma+4,
        17\gamma+1\}}(\Omega;\R^d)}
        \Big)
        ,\\
        \big\|
        \E[
        \mathcal{R}_{-k\tau+(j+1)h}
        | \mathcal{F}
        _{-k\tau+jh}
        ]\big\|_{L^2(\Omega;\R^d)}
        &\leq
         C h^2
        \Big(
        1+\sup_{k\in \mathbb{N}} \sup_{t \geq -k\tau}
        \big\| 
        X_{t}^{-k\tau} 
        \big\|
        ^{5\gamma}
        _{L^{10\gamma}(\Omega;\mathbb{\R}^d)}
        \Big).
    \end{split}
    \end{equation}
\end{lem}

\begin{proof}
    [Proof of Lemma \ref{lem_PMM:the_esti_R_j+1}]
    Recalling the definition of $
    \mathcal{R}_{-k\tau+(j+1)h}$ and using a triangle inequality yield
    \begin{equation}
    \label{eq_PMM:R_j+1}
    \begin{split}
        &
        \|
        \mathcal{R}_{-k\tau+(j+1)h} 
        \|_{L^{2}(\Omega;\R^d)}\\
        & \leq
        \bigg\|
        \int_{-k\tau+jh}^{-k\tau+(j+1)h}
        A\big(
          X_{s}^{-k\tau}
          -
          \Phi
          (X_{-k\tau+jh}^{-k\tau})
        \big)
        \, \dd s 
        \bigg\|
        _{L^{2}(\Omega;\R^d)}\\
        &\quad +
        \bigg\|
        \int_{-k\tau+jh}^{-k\tau+(j+1)h}
          f
        \big(
          s,
         X_{s}^{-k\tau}
          \big)
          -
          f
         \big(
          jh,
         \Phi
         (X_{-k\tau+jh}^{-k\tau})
          \big)
        \, \dd s 
        \bigg\|
        _{L^{2}(\Omega;\R^d)}\\
        &\quad +
        \bigg\|
        \int_{-k\tau+jh}^{-k\tau+(j+1)h}
          g
         \big(
          s,
         X_{s}^{-k\tau}
          \big)
          -
          g
         \big(
          jh,
         \Phi
         (X_{-k\tau+jh}^{-k\tau})
          \big)
        \, \dd W_{s}
        \\
        & \qquad \quad  
        - 
        \sum_{r_1,r_2=1}^{m}
        \mathcal{L}^{r_1}g_{r_2}
        \big(jh,
        \Phi(X_{-k\tau+jh}^{-k\tau})
        \big)
        \Pi^{
        -k\tau+jh,
        -k\tau+(j+1)h}
        _{r_1,r_2}
        \bigg\|_{L^{2}(\Omega;\R^d)}\\
        & \quad +
        \|
        X_{-k\tau+jh}^{-k\tau}
        -
        \Phi
        (X_{-k\tau+jh}^{-k\tau})
        \|_{L^{2}(\Omega;\R^d)}\\
        & :=
        \mathbb{I}_1
        +\mathbb{I}_2
        +\mathbb{I}_3
        +\mathbb{I}_4.
    \end{split}
\end{equation}
    For the term $\mathbb{I}_1$,
    using the H\"older inequality
    and  
    Lemma  \ref{lem_PMM:the_esti_(X(t1)-PhiX(t2)} shows
    \begin{equation}
    \begin{split}
        \mathbb{I}_1
        & \leq
        \int_{-k\tau+jh}^{-k\tau+(j+1)h}
        \|
        A(
        X_s^{-k\tau}
        -
        \Phi(X_{-k\tau+jh}^{-k\tau})
        )
        \|_{L^{2}(\Omega;\R^d)}
        \,\dd s \\
        & \leq
        C h^{\frac{3}{2}}
        \Big(
          1+\sup_{k\in \mathbb{N}} \sup_{t \geq -k\tau}
          \big\| 
          X_{t}^{-k\tau} 
         \big\|
         ^{4\gamma+1}_{L^{8\gamma+2}(\Omega;\mathbb{\R}^d)}
        \Big).
    \end{split}
    \end{equation}
    For the term $\mathbb{I}_2$,
    \begin{equation}
    \begin{split}
        \mathbb{I}_2
        &\leq
        \int_{-k\tau+jh}^{-k\tau+(j+1)h}
        \big\|
         f
        \big(
          s,
         X_{s}^{-k\tau}
          \big)
          -
          f
         \big(
          jh,
         \Phi
         (X_{-k\tau+jh}^{-k\tau})
          \big)
        \big\|
        _{L^2(\Omega;\mathbb{\R}^d)}
        \,\dd s \\
        & \leq
        \underbrace
        {
        \int_{-k\tau+jh}^{-k\tau+(j+1)h}
        \big\|
         f
        \big(
          s,
         X_{s}^{-k\tau}
          \big)
          -
          f
         \big(
          jh,
         X_s^{-k\tau}
          \big)
        \big\|
        _{L^2(\Omega;\mathbb{\R}^d)}
        \,\dd s 
        }_{=:\mathbb{I}_{21}}\\
        & \quad
        +
        \underbrace
        {
        \int_{-k\tau+jh}^{-k\tau+(j+1)h}
        \big\|
         f
        \big(
          jh,
         X_{s}^{-k\tau}
          \big)
          -
          f
         \big(
          jh,
         \Phi
         (X_{-k\tau+jh}^{-k\tau})
          \big)
        \big\|
        _{L^2(\Omega;\mathbb{\R}^d)}
        \,\dd s
        }_{=:\mathbb{I}_{22}}
        .
    \end{split}
    \end{equation}
    Directly
    using
    \eqref
    {eq_PMM:the_esti_f(t,x)-f(s,x)}
    leads to
    \begin{equation}
        \mathbb{I}_{21}
        \leq
        Ch^2
        \Big(
          1+\sup_{k\in \mathbb{N}} \sup_{t \geq -k\tau}
          \big\| 
          X_{t}^{-k\tau} 
         \big\|
         ^\gamma_{L^{2\gamma}(\Omega;\mathbb{\R}^d)}
        \Big).
    \end{equation}
    For the second term,
    it follows from \eqref{eq_PMM:the_esti_f(t,x)-f(t,y)},
    \eqref{eq_PMM:Phi(x)}
    and Lemma \ref{lem_PMM:the_esti_(X(t1)-PhiX(t2)}
    that  for $q_1:=
    \tfrac{5\gamma}{\gamma-1}$,
    $q_2:=
    \tfrac{5\gamma}{4\gamma+1}$
    $
    \Big(\tfrac{1}{q_1}+\tfrac{1}{q_2}=1
    \Big)
    $,
    \begin{equation}
    \label{eq_PMM:I_22}
    \begin{split}
        \mathbb{I}_{22}
        &\leq
        \int_{-k\tau+jh}^{-k\tau+(j+1)h}
        \Big\|
        C
        \big(1+\|X_s^{-k\tau}\|
        +
        \|
        \Phi(X^{-k\tau}_{-k\tau+jh})
        \|
        \big)^{\gamma-1}
        \|
        X_s^{-k\tau}
        -
        \Phi(X_{-k\tau+jh}^{-k\tau})
        \|
        \Big\|_{L^{2}(\Omega;\R)}
        \,\dd s \\
        & \leq
        \int_{-k\tau+jh}^{-k\tau+(j+1)h}
        C
        \big\|
        \big(1+\|X_s^{-k\tau}\|
        +
        \|X^{-k\tau}_{-k\tau+jh}\|
        \big)
        \big\|_{L^{2q_1(\gamma-1)}(\Omega;\R)}
        ^{\gamma-1}
        \times
        \big\|
        X_s^{-k\tau}
        -
        \Phi
        (
        X^{-k\tau}
        _{-k\tau+jh}
        )
        \big\|_{L^{2q_2}(\Omega;\R^d)}
        \,\dd s \\
        & \leq
        Ch^{\frac{1}{2}}
        \int_{-k\tau+jh}^{-k\tau+(j+1)h}
        \Big(
          1+\sup_{k\in \mathbb{N}} \sup_{t \geq -k\tau}
          \big\| 
          X_{t}^{-k\tau} 
         \big\|
         ^{\gamma-1}_{L^{2q_1(\gamma-1)}(\Omega;\mathbb{\R}^d)}
         \Big)
         \Big(
          1+
          \sup
          _{k\in \mathbb{N}} 
          \sup
          _{t \geq -k\tau}
          \big\| 
          X_{t}^{-k\tau} 
         \big\|
         ^{4\gamma+1}
         _{L^{2q_2
         (\gamma+1)
         } (\Omega;
         \mathbb{\R}^d)
         }
         \Big)
         \,\dd s\\
         & \leq
         C h^{\frac{3}{2}}
        \Big(
          1+\sup_{k\in \mathbb{N}} \sup_{t \geq -k\tau}
          \big\| 
          X_{t}^{-k\tau} 
         \big\|
         ^{5\gamma}
         _{L^{10\gamma}
        (\Omega;\mathbb{\R}^d)
         }
        \Big).
    \end{split}
    \end{equation}
    According to Lemma
    \ref{lem_PMM:the_esti_(X(t1)-PhiX(t2)} with
    $p=\frac{10\gamma}{4\gamma+1}$,
    we need to ensure 
    $1 
    \leq
    \frac{10\gamma}{4\gamma+1} 
    \leq \frac{2p^{*}}{4\gamma+1}$,
     $\gamma \in \big[ 1, \frac{p^*}{5} \big)$.
    Therefore,
    \begin{equation}
        \mathbb{I}_2
        \leq
        C h^{\frac{3}{2}}
        \Big(
          1+\sup_{k\in \mathbb{N}} \sup_{t \geq -k\tau}
          \big\| 
          X_{t}^{-k\tau} 
         \big\|
         ^{5\gamma}_{L^{10\gamma}(\Omega;\mathbb{\R}^d)}
        \Big).
    \end{equation}
    In view of the It\^o isometry, we get
    \begin{equation}
    \label{eq_PMM:|I_3|^2}
    \begin{split}
        |\mathbb{I}_3|^2 
        &=
        \sum_{r_2=1}^m        \int_{-k\tau+jh}^{-k\tau+(j+1)h}
        \E\Big[
        \Big\|
          g_{r_2}
         \big(
          s,
         X_{s}^{-k\tau}
          \big)
          -
          g_{r_2}
         \big(
          jh,
         \Phi
         (X_{-k\tau+jh}^{-k\tau})
          \big) \\
          & 
          \qquad \qquad \qquad
          \qquad \qquad 
          - 
          \sum_{r_1=1}^{m}
        \mathcal{L}^{r_1}g_{r_2}
        \big(jh,
        \Phi(X_{-k\tau+jh}^{-k\tau})
        \big)
        (W_s^{r_1}-W_{t_j}^{r_1})
        \Big\|^2
        \Big]
        \,\dd s \\
        & \leq
        2\sum_{r_2=1}^m 
        \int_{-k\tau+jh}^{-k\tau+(j+1)h}
        \E \Big[
        \Big\|
          g_{r_2}
         \big(
          s,
         X_{s}^{-k\tau}
          \big)
          -
          g_{r_2}
         \big(
          jh,
          X_{s}^{-k\tau}
          \big)
          \Big\|^2 
          \Big]
          \,\dd s \\
         & \quad +
        2 \sum_{r_2=1}^m
        \int_{-k\tau+jh}^{-k\tau+(j+1)h}
        \E\Big[
        \Big\|
          g_{r_2}
         \big(
          jh,
         X_{s}^{-k\tau}
          \big) 
          -
          g_{r_2}
         \big(
          jh,
         \Phi
         (X_{-k\tau+jh}^{-k\tau})
          \big) \\
          & 
          \qquad  \qquad
          \qquad \qquad
          \qquad \quad
          - 
          \sum_{r_1=1}^{m}
          \mathcal{L}^{r_1}g_{r_2}
          \big(
          jh,
         \Phi
         (X_{-k\tau+jh}^{-k\tau})
          \big)
        (W_s^{r_1}-W_{t_j}^{r_1})
        \Big\|^2
        \Big]
        \,\dd s \\
        & :=
        \mathbb{I}_{31}
        +
        \mathbb{I}_{32}.
    \end{split}
    \end{equation}
    For the term $\mathbb{I}_{31}$, we denote $\mathbb{T}_1
         :=
        \Big\|
          g_{r_2}
         \big(
          s,
         X_{s}^{-k\tau}
          \big) 
          -
          g_{r_2}
         \big(
          jh,
         X_{s}^{-k\tau})
          \big)
        \Big\|
        $.
    For $s \in [jh,(j+1)h)$ and $r_2 \in [m]$,
    using \eqref{eq_PMM:the_esti_g(t,x)-g(s,x)} shows
   \begin{equation}
    \begin{split}
         \|
         \mathbb{T}_1
         \|
         ^2
         _{L^2(\Omega;
         \R^d)}
          & \leq
          \Big\|
          C
          \big(
          1+
          \|
          X_{s}
          ^{-k\tau}
          \|
          \big)
          ^{\frac{\gamma+1}{2}}
          _{L^{\gamma+1}(\Omega;\R^d)}
          |s-jh|
          \Big\|
          ^2
          \\
          & \leq
          Ch^2
          \Big(
          1+\sup_{k \in \N}
          \sup_{t\geq -k\tau}
          \|X_{t}^{-k\tau}\|
          ^{\gamma+1}
          _{L^{2(\gamma+1)}(\Omega;\R^d)}
          \Big).
    \end{split}
   \end{equation}
   As a result,
   one can see that
   \begin{equation}
   \label{eq_PMM:I_31}
       \mathbb{I}_{31}
       \leq
       Ch^3
          \Big(
          1+\sup_{k \in \N}
          \sup_{t\geq -k\tau}
          \|X_{t}^{-k\tau}\|
          ^{\gamma+1}
          _{L^{2(\gamma+1)}(\Omega;\R^d)}
          \Big).
   \end{equation}
   Next, we treat the term $\mathbb{I}_{32}$ and
   denote
   \begin{equation}
       \begin{split}
           \mathbb{T}_2
           & :=
           \underbrace{
            g_{r_2}
            \big(
            j h,
            X_{s}^{-k\tau}
            \big) 
            -
            g_{r_2}
            \big(
            j h,
            \Phi
            (
            X_{-k\tau+jh}
            ^{-k\tau}
            )
            \big)
            }_{=:\mathbb{T}_{21}}
            - \sum_{r_1=1}
            ^{m}
            \mathcal{L}
            ^{r_1}g_{r_2}
            \big(
            jh,
            \Phi
            (
            X_{-k\tau+jh}
            ^{-k\tau}
            )
            \big)
            (
            W_s^{r_1}
            -
            W_{t_j}^{r_1}
            ).
       \end{split}
   \end{equation}
    Applying the mean value theorem yields
   \begin{equation}
   \label{eq_PMM:g(t,X)-g(t,Phi(x)}
    \begin{split}
           \mathbb{T}
           _{21}
           &= 
          \tfrac{\partial g_{r_2}}{\partial x}
          \big(
          jh,
         \Phi
         (X_{-k\tau+jh}^{-k\tau})
          \big)
          \big(
          X_{s}^{-k\tau}
          -
          \Phi
          (
          X_{-k\tau+jh}
          ^{-k\tau}
          )
          \big)
          +
          \mathcal{R}
          _{g_{r_2}} \\
          &=
           \tfrac
           {\partial g_{r_2}}{\partial x}
            \big(
            jh,
            \Phi
            (X_{-k\tau+jh}^{-k\tau})
            \big)
            \big(
             X_{s}^{-k\tau}
             -
             X_{-k\tau+jh}
             ^{-k\tau}
             +
             X_{-k\tau+jh}
             ^{-k\tau}
             -
             \Phi
             (
             X_{-k\tau+jh}
             ^{-k\tau}
             )
           \big)
            +
            \mathcal{R}
            _{g_{r_2}} \\
            & =
            \tfrac{\partial g_{r_2}}{\partial x}
            \big(
            jh,
            \Phi
            (
            X_{-k\tau+jh}
            ^{-k\tau}
            )
            \big)
            \Big(
            \int
            _{-k\tau+jh}^s
            A
            (X_{\rho}
            ^{-k\tau})
            )
            \,\dd \rho
            +
            \int
            _{-k\tau+jh}^s
            f
            (
            \rho,
            X_{\rho}
            ^{-k\tau})
            )
            \,\dd \rho 
            \\
            & 
            \quad +
            \int_{k\tau+jh}^s
            g
            (
            \rho,
            X_{\rho}^{-k\tau})
            \,\dd W_{\rho}
            +
            X_{-k\tau+jh}
             ^{-k\tau}
             -
             \Phi
             (
             X_{-k\tau+jh}
             ^{-k\tau}
             )
            \Big)
            +
            \mathcal{R}
            _{g_{r_2}},
    \end{split}
   \end{equation}
   where for short we denote
   \begin{equation}
    \begin{split}
        \mathcal{R}_{g_{r_2}}
        &:=
        \int_0^1
        \Big(
        \tfrac{\partial g_{r_2}}{\partial x}
          \big(
          jh,
         \Phi
         (X_{-k\tau+jh}^{-k\tau})
         +l
         (X_s^{-k\tau}
         -\Phi
         (X_{-k\tau+jh}^{-k\tau}))
          \big)
          -
          \tfrac{\partial g_{r_2}}{\partial x}
          \big(
          jh,
         \Phi
         (X_{-k\tau+jh}^{-k\tau})
         \big)
         \Big)\\
         & \qquad \times
         (X_s^{-k\tau}
         -\Phi
         (X_{-k\tau+jh}^{-k\tau}))
         \,\dd l.
    \end{split}
   \end{equation}
   Now it follows from \eqref{eq_PMM:g(t,X)-g(t,Phi(x)} that
   \begin{equation}
    \begin{split}
        & \|
        \mathbb{T}_2
        \|
        _{L^2(\Omega;\R^d)} \\
        & \leq
        \underbrace
        {
        \Big\|
        \tfrac
        {\partial g_{r_2}}{\partial x}
            \big(
            jh,
            \Phi
            (
            X_{-k\tau+jh}
            ^{-k\tau}
            )
            \big)
            \int_{-k\tau+jh}^s
            A
            X_{\rho}^{-k\tau}
            \,\dd  \rho
        \Big\|
        _{L^2(\Omega;\R^d)} 
        }_{=:\mathbb{B}_1}\\
        & \quad +
         \underbrace
         {
         \Big\|
         \tfrac
         {\partial g_{r_2}}{\partial x}
            \big(
            jh,
            \Phi
            (
            X_{-k\tau+jh}
            ^{-k\tau}
            )
           \big)
           \int_{-k\tau+jh}^s
            f(
            \rho,
            X_{\rho}^{-k\tau}
            )
            \,\dd  \rho
        \Big\|
        _{L^2(\Omega;\R^d)}
         }_{=:\mathbb{B}_2}
        +
        \|
        \mathcal{R}_{g_{r_2}}
        \|
        _{L^2(\Omega;\R^d)}\\
        & \quad +
        \underbrace
        {
        \sum_{r_1=1}^m
        \Big\|
          \tfrac
          {\partial g_{r_2}}{\partial x}
          \Big( 
           jh,
           \Phi
           (
           X_{-k\tau+jh}^{-k\tau}
           )
           \Big)
           \int_{-k\tau+jh}^s
           \Big(
            g_{r_1}
            \big(
            \rho,
            (X_{\rho}^{-k\tau})
            \big)
            -
            g_{r_1}
            \big(
              jh,
              \Phi
              (
              X_{-k\tau+jh}^{-k\tau}
              )
           \big)
           \Big)
        \,\dd W_{\rho}^{r_1}
    \Big\|
    _{L^2(\Omega;\R^d)}
     }_{=:\mathbb{B}_3}\\
     & \quad +
     \underbrace
     {
        \Big\|
             \tfrac
              {\partial g_{r_2}}{\partial x}
              \big(
             jh,
             \Phi
             (
              X_{-k\tau+jh}
              ^{-k\tau}
              )
            \big)
            \big(
             X_{-k\tau+jh}
             ^{-k\tau}
             -
             \Phi
             (
             X_{-k\tau+jh}
             ^{-k\tau}
             )
             \big)
        \Big\|
    _{L^2(\Omega;\R^d)}
     }_{=:\mathbb{B}_4}.
    \end{split}
    \end{equation}
   In the following we cope with the above 
   five items separately.
   According to \eqref{eq_PMM:the_esti_g'(t,x)} and the H\"older inequality,
    one can see that
   \begin{equation}
       \begin{split}
        \mathbb{B}_1
         & 
         \leq
         C
         \int_{-k\tau+jh}^s
         \Big\|
         \big(
            1+
            \|
            \Phi
            (
            X_{-k\tau+jh}
            ^{-k\tau}
            )
            \|
         \big)
         ^{\frac{\gamma-1}{2}}
            \big\|
             A
             X_{\rho}^{-k\tau}
            \big\|
         \Big\|
         _{L^2(\Omega;\R^d)} 
         \,\dd \rho\\
         & 
         \leq
         C
         \int_{-k\tau+jh}^s
         \Big\|
         \big(
            1+
            \|
            X_{-k\tau+jh}
            ^{-k\tau}
            \|
         \big)
         ^{\frac{\gamma-1}{2}}
         \|
        A X_{\rho}^{-k\tau}
         \|
         \Big\|
         _{L^2(\Omega;\R^d)} 
         \,\dd \rho\\
         & \leq
         Ch
         \Big(
          1
          +
          \sup_{k \in \N}
          \sup_{t\geq -k\tau}
          \|X_{t}^{-k\tau}\|
          ^{\frac{\gamma+1}{2}}
          _{L^{\gamma+1}(\Omega;\R^d)}
          \Big).
    \end{split}
   \end{equation}
   Owing to \eqref{eq_PMM:the_esti_g'(t,x)},
   \eqref{eq_PMM:the_esti_f(t,x)} and the H\"older inequality,
   \begin{equation}
    \begin{split}
         \mathbb{B}_2
         & \leq
         C
         \int_{-k\tau+jh}^s
         \Big\|
         \big(
         1+
         \|
         \Phi
         (X_{-k\tau+jh}^{-k\tau})
         \|
         \big)
         ^{\frac{\gamma-1}{2}}
         \big\|
         f(
         \rho,
         X_{\rho}^{-k\tau}
         )
         \big\|
         \Big\|_{L^2(\Omega;\R^d)} 
         \,\dd \rho\\
         & \leq
         C
         \int_{-k\tau+jh}^s
         \Big\|
         \big(1+
         \|
         X_{-k\tau+jh}^{-k\tau}
         \|
         \big)^{\frac{\gamma-1}{2}}
         \big(
         1+
         \|
         X_{\rho}^{-k\tau})\|
         \big)^{\gamma}
         \Big\|
         _{L^2(\Omega;\R^d)} 
         \,\dd \rho\\
         & \leq
         Ch
         \Big(
          1+\sup_{k \in \N}
          \sup_{t\geq -k\tau}
          \|X_{t}^{-k\tau}\|
          ^{\frac{3\gamma-1}{2}}_{L^{3\gamma-1}(\Omega;\R^d)}
          \Big).
    \end{split}
   \end{equation}
   Armed with the condition \eqref{eq_PMM:the_esti_g'(t,x)-g'(t,y)},
   \eqref{eq_PMM:Phi(x)}
   and   
   Lemma
   \ref{lem_PMM:the_esti_(X(t1)-PhiX(t2)},
   one can further use the H\"older inequality to acquire
   \begin{equation} 
    \begin{split}
        &\|
        \mathcal{R}_{g_{r_2}}
        \|_{L^2(\Omega;\R^d)} \\
        &\leq
        \int_0^1
        \Big\|
        \Big[
        \tfrac{\partial g_{r_2}}{\partial x}
          \big(
          jh,
         \Phi
         (X_{-k\tau+jh}^{-k\tau})
         +l
         (X_s^{-k\tau}
         -\Phi
         (X_{-k\tau+jh}^{-k\tau}))
          \big)
          -
          \tfrac{\partial g_{r_2}}{\partial x}
          \big(
          jh,
         \Phi
         (X_{-k\tau+jh}^{-k\tau})
         \big)
         \Big]\\
         & \quad \times
         (X_s^{-k\tau}
         -\Phi
         (X_{-k\tau+jh}^{-k\tau}))
         \Big\|         _{L^2(\Omega;\R^d)}
         \,\dd l \\
         & \leq
         C
         \int_0^1
         \Big\|
         \big(
         1+
         \|l X_s^{-k\tau}
         +(1-l)
         \Phi
         (X_{-k\tau+jh}^{-k\tau})\|
         +
         \|\Phi
         (X_{-k\tau+jh}^{-k\tau})\|
         \big)^
         {\max\big\{\frac{\gamma-3}{2},0\big\}}\\
         & \quad \times
         \|X_s^{-k\tau}
         -\Phi
         (X_{-k\tau+jh}^{-k\tau})\|^2
         \Big\|_{L^2(\Omega;\R^d)}
         \,\dd l \\
         & \leq
         Ch
         \Big(
          1+\sup_{k \in \N}
          \sup_{t\geq -k\tau}
          \|X_{t}^{-k\tau}\|
          ^{
          \frac{\max\{16\gamma+4,
          17\gamma+1
          \}}{2}}_{L^{\max\{16\gamma+4,
          17\gamma+1\}}(\Omega;\R^d)}
          \Big).
    \end{split}
   \end{equation}
   According to Lemma
    \ref{lem_PMM:the_esti_(X(t1)-PhiX(t2)} with
    $p=\frac{\max\{16\gamma+4,17\gamma+1\}}{4\gamma+1}$,
    we need to ensure
    $1 <
    \frac{\max\{16\gamma+4,17\gamma+1\}}{4\gamma+1}
    \leq
    \frac{2p^*}{4\gamma+1}$,
    $\gamma \in [1,\frac{2p*-1}{17}\big)$.
   Again, using the It\^o isometry and the H\"older inequality gives
   \begin{equation}
        \mathbb{B}_3
         =
         \sum_{r_1=1}^m
         \Big(
         \int_{-k\tau+jh}^s
         \Big\|
          \tfrac
          {\partial g_{r_2}}{\partial x}
          \Big(
            jh,
            \Phi
            (X_{-k\tau+jh}
            ^{-k\tau}
          \Big)
          \Big(
            g_{r_1}
            (
            \rho,
            X_{\rho}^{-k\tau}
            )
            -
            g_{r_1}
            \big(
            jh,
           \Phi
            (X_{-k\tau+jh}^{-k\tau})
           \Big)
           \Big\|^2
           _{L^2(\Omega;\R^d)}
            \,\dd \rho
            \Big)
            ^{\frac{1}{2}}.
   \end{equation}
   Now, 
   it follows from
   \eqref{eq_PMM:the_esti_g(t,x)-g(t,y)} and \eqref{eq_PMM:the_esti_g(t,x)-g(s,x)} 
   that
   \begin{equation}
    \begin{split}
        &\Big\|
           g_{r_1}
            \big(
            \rho,
            X_{\rho}^{-k\tau}
            \big)
           -
            g_{r_1}
            \big(
            jh,
            \Phi
            (X_{-k\tau+jh}^{-k\tau})
            \big)
        \Big\| \\
        & \leq
         \Big\|
         g_{r_1}
            \big(
             \rho,
             X_{\rho}^{-k\tau}
            \big)
         -
        g_{r_1}
            \big(
            \rho,
            \Phi
            (X_{-k\tau+jh}^{-k\tau})
            \big)
         \Big\|
          +
          \Big\|
           g_{r_1}
            \big(
            \rho,
            \Phi
            (X_{-k\tau+jh}^{-k\tau})
            \big)
          -
           g_{r_1}
            \big(
            jh,
            \Phi
            (
            X_{-k\tau+jh}^{-k\tau}
            )
            \big)
          \Big\| \\
          & \leq
            C
            \Big(
            1+
            \Big\|
            X_{\rho}^{-k\tau}
            \Big\|
            +
            \Big\|
            \Phi
            (X_{-k\tau+jh}
            ^{-k\tau})
            \Big\|
            \Big)
            ^{{\frac{\gamma-1}{2}}}
            \Big\|
            X_{\rho}^{-k\tau}
            -
            \Phi
            (
            X_{-k\tau+jh}
            ^{-k\tau})
            \Big\| \\
         & \quad +
         C
         \Big(
         1+
         \Big\|\Phi
         (X_{-k\tau+jh}^{-k\tau})\Big\|
         \Big)^{\frac{\gamma+1}{2}}
         |\rho-jh|.
    \end{split}
   \end{equation}
   With the help of \eqref{eq_PMM:the_esti_g'(t,x)} and
   Lemma \ref{lem_PMM:Phi(x)},
   one can infer that
   \begin{equation}
    \begin{split}
        &\Big\|
          \tfrac
          {\partial g_{r_2}}{\partial x}
          \Big(
            jh,
            \Phi
            ( 
            X_{-k\tau+jh} ^{-k\tau}
            )
          \Big)
          \Big(
            g_{r_1}
            \big(
            \rho,
            X_{\rho}^{-k\tau}
            \big)
            -
            g_{r_1}
            \big(
            jh,
            \Phi
            (
            X_{-k\tau+jh}
            ^{-k\tau}
            )
            \big)
          \Big)
          \Big\|
          _{L^2(\Omega;\R^d)} \\
          & \leq
          C
          \Big\|
          \big(
           1+
           \|\Phi
           (
           X_{-k\tau+jh}
           ^{-k\tau}
           )
           \|
           \big)
           ^{\frac{\gamma-1}{2}}
           \big\|
            g_{r_1}
            \big(
            \rho,
            X_{\rho}^{-k\tau}
            \big)
         -
         g_{r_1}
         \big(
          jh,
         \Phi
         (X_{-k\tau+jh}^{-k\tau})
          \big)
        \big\|
          \Big\|_{L^2(\Omega;\R)}\\
          & \leq
          C
          \Big\|
          \big(
           1+
           \|\Phi
           (
           X_{-k\tau+jh}
           ^{-k\tau}
           )
           \|
           \big)
           ^{\frac{\gamma-1}{2}}
           \cdot
            \big(
             1+
             \|
              X_{\rho}^{-k\tau}
             \|
             +
            \|\Phi
            (X_{-k\tau+jh}^{-k\tau})\|
           \big)^
           \frac{\gamma-1}{2}
           \cdot
            \|
           X_{\rho}^{-k\tau}
            -
           \Phi
            (X_{-k\tau+jh}^{-k\tau})
            \|
         \Big\|_{L^2(\Omega;\R)}\\
         & \quad +
         C
         \Big\|
         \big(
         1+
         \|
         \Phi
         (X_{-k\tau+jh}
         ^{-k\tau})\|
         \big)
         ^{\frac{\gamma-1}{2}}
         \big(
         1+
         \|\Phi
         (X_{-k\tau+jh}^{-k\tau})\|
         \big)
         ^{\frac{\gamma+1}{2}}
        \Big\|
        _{L^2(\Omega;\R)}
         |\rho-jh| \\
         & \leq
         \underbrace
         {
          C
          \Big\|
          \big(
          1+
          \|
         X_{\rho}^{-k\tau}\|
         +
         \|
         X_{-k\tau+jh}^{-k\tau}\|
         \big)^{\gamma-1}
         \|
         X_{\rho}^{-k\tau}
         -
         \Phi
         (
         X_{-k\tau+jh}^{-k\tau}
         )
         \|
         \Big\|_{L^2(\Omega;\R)}
         }_{=:\mathbb{M}}
         \\
         & \quad +
         C
         \Big\|
         \big(
         1+
         \|
         X_{-k\tau+jh}^{-k\tau}\|
         \big)^{\gamma}
         \Big\|
         _{L^2(\Omega;\R)}
         |\rho-jh|. \\
    \end{split}
   \end{equation}
   Similarly,
   recalling \eqref{eq_PMM:I_22}
   leads to
   \begin{equation}
       \begin{split}
           \mathbb{M}
            & \leq
            C
            \big\|
            \big(
            1+
            \|
            X_{\rho}^{-k\tau}
            \|
            +
            \|
            X^{-k\tau}
            _{-k\tau+jh}
            \|
            \big)
            \big\|
            _{L^{2q_1(\gamma-1)}
            (\Omega;\R)}
            ^{\gamma-1}
            \times
            \big\|
            X_{\rho}^{-k\tau}
            -
            \Phi(
            X^{-k\tau+jh}
            _{-k\tau+jh})
             \big\|
             _{L^{2q_2}(\Omega;\R^d)}
             \\
             & \leq
            Ch^{\frac{1}{2}}
             \Big(
              1
              +
              \sup_{k\in \mathbb{N}}
              \sup_{t \geq -k\tau}
               \big\| 
               X_{t}^{-k\tau} 
               \big\|
               ^{\gamma-1}
               _{L^{2q_1(\gamma-1)}(\Omega;\mathbb{\R}^d)}
              \Big)
              \Big(
              1+
              \sup
              _{k\in \mathbb{N}} 
              \sup
              _{t \geq -k\tau}
              \big\| 
              X_{t}^{-k\tau} 
              \big\|
              ^{4\gamma+1}
              _{L^{2q_2
              (\gamma+1)(\Omega;
              \mathbb{\R}^d)}
              }
              \Big)
              \,\dd s\\
              & \leq
              C h^{\frac{1}{2}}
              \Big(
               1
               +
               \sup_{k\in \mathbb{N}} 
               \sup_{t \geq -k\tau}
               \big\| 
               X_{t}^{-k\tau} 
               \big\|
               ^{5\gamma}
               _{L^{10\gamma}
               (\Omega;
               \mathbb{\R}^d)}
               \Big).
       \end{split}
   \end{equation}
   with $q_1=\frac{5\gamma}{\gamma-1},
   q_2=\frac{5\gamma}{4\gamma+1}(\frac{1}{q_1}+\frac{1}{q_2}=1)$.
   According to Lemma \ref{lem_PMM:the_esti_(X(t1)-PhiX(t2)},
    $p
    =\frac{10\gamma}{4\gamma+1}$,
    we need to ensure 
    $1 
    \leq
    \frac{10 \gamma}{4\gamma+1} 
    \leq 
    \frac{2p^{*}}{4\gamma+1}$,
    $\gamma \in \big[
    1,\frac{p^{*}}{5}
    \big]$.
   All in all, one can deduce
   \begin{equation}
       \mathbb{B}_3
       \leq
       C h^{\frac{3}{2}}
         \Big(
          1+\sup_{k \in \N}
          \sup_{t\geq -k\tau}
          \|X_{t}^{-k\tau}\|
          ^{5\gamma}
          _{L^{10\gamma}(\Omega;\R^d)}
          \Big).
   \end{equation}
    For the term $\mathbb{B}_4$,
   one can employ \eqref{eq_PMM:the_esti_g'(t,x)},
   Lemma \ref{lem_PMM:x-Phi(x)},
    \eqref{eq_PMM:Phi(x)}
    and H\"older inequality
   to get
   \begin{equation}
       \begin{split}
           \mathbb{B}_4
           & \leq
           C h^2
           \Big\|
           \big(
           1+
           \|
           \Phi
           (
           X_{-k\tau+jh}^{-k\tau}
           )
           \|
           \big)
           ^{\frac{\gamma-1}{2}}
           \|
           X_{-k\tau+jh}^{-k\tau}
           \|^{4\gamma+1}
           \Big\|
           _{L^2(\Omega;\R^d)} \\
           & \leq
           C h^2
           \Big\|
           \big(
           1+
           \|
           X_{-k\tau+jh}^{-k\tau}
           \|
           \big)
           ^{\frac{\gamma-1}{2}}
           \|
           X_{-k\tau+jh}^{-k\tau}
           \|^{4\gamma+1}
           \Big\|
           _{L^2(\Omega;\R^d)} \\
           & \leq
           C h^2
           \big(
           1+
           \|
           X_{-k\tau+jh}^{-k\tau}
           \|
           \big)
           ^{\frac{\gamma-1}{2}}
           _{L^{v_1(\gamma-1)}(\Omega;\R)}
           \|
           X_{-k\tau+jh}^{-k\tau}
           \|^{4\gamma+1}
           _{L^{2v_2(4\gamma+1)}(\Omega;\R^d)},
       \end{split}
   \end{equation}
   where we take $v_1=\frac{9\gamma+1}{\gamma-1}$,
   $v_2=\frac{9\gamma+1}{8\gamma+2}$ such that $v_1(\gamma-1)=2v_2(4\gamma+1)$, $\frac{1}{v_1}+\frac{1}{v_2}=1$
   and one can then arrive at 
   \begin{equation}
       \mathbb{B}_4
       \leq
       Ch^2
         \Big(
          1
          +
          \sup_{k \in \N}
          \sup_{t\geq -k\tau}
          \|X_{t}^{-k\tau}\|
          ^{\frac{9\gamma+1}{2}}
          _{L^{9\gamma+1}(\Omega;\R^d)}
          \Big).
   \end{equation}
   To sum up,
   one can obtain
   \begin{equation}
   \label{eq_PMM:I_32}
       \mathbb{I}_{32}
       \leq
        Ch
        \Big(
        1+\sup_{k \in \N}
        \sup_{t\geq -k\tau}
        \|X_{t}^{-k\tau}\|
        ^{
        \frac{\max\{16\gamma+4,
        17\gamma+1
        \}}{2}}_{L^{\max\{16\gamma+4,
        17\gamma+1\}}(\Omega;\R^d)}
        \Big).
   \end{equation}
   This together with
   \eqref{eq_PMM:|I_3|^2} and
   \eqref{eq_PMM:I_31} yields
   \begin{equation}
       \mathbb{I}_3
       \leq
       Ch^{\frac{3}{2}}
       \Big(
        1+\sup_{k \in \N}
        \sup_{t\geq -k\tau}
        \|X_{t}^{-k\tau}\|
        ^{
        \frac{\max\{16\gamma+4,
        17\gamma+1
        \}}{2}}_{L^{\max\{16\gamma+4,
        17\gamma+1\}}(\Omega;\R^d)}
        \Big).
   \end{equation}
   With regard to $\mathbb{I}_4$,
   we use Lemma \ref{lem_PMM:x-Phi(x)} to show
   \begin{equation}
       \mathbb{I}_4
       \leq
       Ch^2
       \|
       X_{-k\tau+jh}
       ^{-k\tau}
       \|
       ^{4\gamma+1}
       _{L^{8\gamma+2}(\Omega;\R^d)}.
   \end{equation}
   Putting all the above estimates together we derive from \eqref{eq_PMM:R_j+1} that
   \begin{equation}
        \|
        \mathcal{R}_{-k\tau+(j+1)h} 
        \|_{L^{2}(\Omega;\R^d)}
        \leq
        Ch^{\frac{3}{2}}
         \Big(
        1+\sup_{k \in \N}
        \sup_{t\geq -k\tau}
        \|X_{t}^{-k\tau}\|
        ^{
        \frac{\max\{16\gamma+4,
        17\gamma+1
        \}}{2}}_{L^{\max\{16\gamma+4,
        17\gamma+1\}}(\Omega;\R^d)}
        \Big).
   \end{equation}
   Noting that the stochastic integral vanishes under the conditional expectation, we derive
   \begin{equation}
   \label{eq_PMM:R_j+1|F}
    \begin{split}
        &\big\|
        \E[
        \mathcal{R}_{-k\tau+(j+1)h}
        | \mathcal{F}
        _{-k\tau+jh}
        ]\big\|_{L^2(\Omega;\R^d)}\\
        & \leq
        \bigg\|
        \E\bigg[
        \int_{-k\tau+jh}^{-k\tau+(j+1)h}
        A\big(
          X_{s}^{-k\tau}
          -
          \Phi
          (X_{-k\tau+jh}^{-k\tau})
        \big)
        \, \dd s 
        \Big|\mathcal{F}_{-k\tau+jh}
        \bigg]
        \bigg\|_{L^2(\Omega;\R^d)}\\
        & \quad +
        \bigg\|
        \E\bigg[
        \int_{-k\tau+jh}^{-k\tau+(j+1)h}
          f
        \big(
          s,
         X_{s}^{-k\tau}
          \big)
          -
          f
         \big(
          jh,
         \Phi
         (X_{-k\tau+jh}^{-k\tau})
          \big)
        \, \dd s 
        \Big|\mathcal{F}_{-k\tau+jh}
        \bigg]
        \bigg\|_{L^2(\Omega;\R^d)}\\
        & \quad +
        \|
        X_{-k\tau+jh}^{-k\tau}
        -
        \Phi
        (X_{-k\tau+jh}^{-k\tau})
        \|_{L_2(\Omega;\R^d)} \\
        & :=
        \mathbb{I}_5
        +
        \mathbb{I}_6
        +
        \mathbb{I}_7.
    \end{split}
   \end{equation}
   In order to estimate $\mathbb{I}_5$,
   we first note that
   \begin{equation}
    \begin{split}
        &\E
        \bigg[
        \int_{-k\tau+jh}^{-k\tau+(j+1)h}
        \int_{-k\tau+jh}^{s}
        g
        \big(
        r,
        X_{r}^{-k\tau}
        \big)
        \, \dd W_{r}
        \, \dd s
        \Big|\mathcal{F}_{-k\tau+jh}
        \bigg] \\
        &=
        \int_{-k\tau+jh}^{-k\tau+(j+1)h}
        \E
        \bigg[
        \int_{-k\tau+jh}^{s}
        g
        \big(
        r,
        X_{r}^{-k\tau}
        \big)
        \, \dd W_{r}
        \Big|\mathcal{F}_{-k\tau+jh}
        \bigg] 
        \, \dd s\\
        & =
        0.
    \end{split}
   \end{equation}
   As a result,
   one infers that
   \begin{equation}
    \begin{split}
        &
        \E\bigg[
        \int_{-k\tau+jh}^{-k\tau+(j+1)h}
        A\big(
          X_{s}^{-k\tau}
          -
          \Phi
          (X_{-k\tau+jh}^{-k\tau})
        \big)
        \, \dd s 
        \Big|\mathcal{F}_{-k\tau+jh}
        \bigg]
        \\
        &=
        \E\bigg[
        \int_{-k\tau+jh}^{-k\tau+(j+1)h}
            A\big(
                X_{s}^{-k\tau}
                -
                X_{-k\tau+jh}
                ^{-k\tau}
                +
                X_{-k\tau+jh}
                ^{-k\tau}
                -
                \Phi
                (X_{-k\tau+jh}
                ^{-k\tau})
            \big)
        \, \dd s 
        \Big|\mathcal{F}_{-k\tau+jh}
        \bigg]\\
        & =
        \E
        \bigg[
            \int_{-k\tau+jh}
            ^{-k\tau+(j+1)h}
            \int_{-k\tau+jh}^r
            A^2
            X_{r}^{-k\tau}
            +
            A
             f
            \big(
             r,
             X_{r}^{-k\tau}
            \big)
            \,\dd r 
            \,\dd s
            \Big|
            \mathcal{F}_{-k\tau+jh}
        \bigg] \\
        & \quad +
        \E\bigg[
        \int_{-k\tau+jh}^{-k\tau+(j+1)h}
            A\big(
                X_{-k\tau+jh}
                ^{-k\tau}
                -
                \Phi
                (X_{-k\tau+jh}
                ^{-k\tau})
            \big)
        \, \dd s 
        \Big|\mathcal{F}_{-k\tau+jh}
        \bigg] .\\
    \end{split}
   \end{equation}
   Based on the Jensen inequality and the H\"older inequality,
   according to \eqref{eq_PMM:the_esti_f(t,x)},
   one can show that
   \begin{equation}
    \begin{split}
        & 
        \bigg\|
        \E
        \bigg[
        \int_{-k\tau+jh}^{-k\tau+(j+1)h}
        \int_{-k\tau+jh}^r
        A^2
        X_{r}^{-k\tau}
        +
        A
        f
        \big(
        r,
        X_{r}^{-k\tau}
        \big)
        \,\dd r 
        \,\dd s
        \Big|\mathcal{F}_{-k\tau+jh}
        \bigg]
        \bigg\|_{L^2(\Omega;\R^d)}\\
        & \leq
        \bigg\|
        \int_{-k\tau+jh}^{-k\tau+(j+1)h}
        \int_{-k\tau+jh}^r
        A^2
        X_{r}^{-k\tau}
        +
        A
        f
        \big(
        r,
        X_{r}^{-k\tau}
        \big)
        \,\dd r 
        \,\dd s
        \bigg\|_{L^2(\Omega;\R^d)} \\
        & \leq
        \int_{-k\tau+jh}^{-k\tau+(j+1)h}
        \int_{-k\tau+jh}^r
        \big\|
        A^2
        X_{r}^{-k\tau}
        +
        A
        f
        \big(
        r,
        X_{r}^{-k\tau}
        \big)
        \big\|
        _{L^2(\Omega;\R^d)}
        \,\dd r 
        \,\dd s \\
        & \leq
                \int_{-k\tau+jh}^{-k\tau+(j+1)h}
        \int_{-k\tau+jh}^r
        \big\|
        A^2
        X_{r}^{-k\tau}
        \big\|_{L^2(\Omega;\R^d)}
        \,\dd r 
        \,\dd s \\
        & \quad +
                \int_{-k\tau+jh}^{-k\tau+(j+1)h}
        \int_{-k\tau+jh}^r
        \big\|
        A
        f
        \big(
        r,
        X_{r}^{-k\tau}
        \big)
        \big\|
        _{L^2(\Omega;\R^d)}
        \,\dd r 
        \,\dd s \\
        & \leq
        C h^2
        \Big(
          1+\sup_{k\in \mathbb{N}} \sup_{t \geq -k\tau}
          \big\| 
          X_{t}^{-k\tau} 
         \big\|
         ^\gamma_{L^{2\gamma}(\Omega;\mathbb{\R}^d)}
        \Big).
    \end{split}
   \end{equation}
   In light of the Jensen inequality,
   the H\"older inequality and 
   Lemma \ref{lem_PMM:x-Phi(x)},
   one can get
   \begin{equation}
    \begin{split}
        & 
        \bigg\|
        \E
        \bigg[
        \int_{-k\tau+jh}^{-k\tau+(j+1)h}
        A
        \big(
        X_{-k\tau+jh}^{-k\tau}
        -
        \Phi
        (
        X_{-k\tau+jh}^{-k\tau}
        )
        \big)
        \,\dd s
        \Big|\mathcal{F}_{-k\tau+jh}
        \bigg]
        \bigg\|_{L^2(\Omega;\R^d)}\\
        & \leq
        \bigg\|
        \int_{-k\tau+jh}^{-k\tau+(j+1)h}
        A
        \big(
        X_{-k\tau+jh}^{-k\tau}
        -
        \Phi
        (
        X_{-k\tau+jh}^{-k\tau}
        )
        \big)
        \,\dd s
        \bigg\|
        _{L^2(\Omega;\R^d)} \\
        & \leq
        \int_{-k\tau+jh}^{-k\tau+(j+1)h}
        \big\|
        A
        \big(
        X_{-k\tau+jh}^{-k\tau}
        -
        \Phi
        (
        X_{-k\tau+jh}^{-k\tau}
        )
        \big)
        \big\|
        _{L^2(\Omega;\R^d)}
        \,\dd s \\
        & \leq
        C 
        h^3
        \|
        X_{-k\tau+jh}
        ^{-k\tau}
        \|
        ^{4\gamma+1}
        _{L^{8\gamma+2}(\Omega;\R^d)}.
    \end{split}
   \end{equation}
   Therefore,
   \begin{equation}
       \mathbb{I}_5
       \leq
        C h^2
        \Big(
          1+\sup_{k\in \mathbb{N}} \sup_{t \geq -k\tau}
          \big\| 
          X_{t}^{-k\tau}
         \big\|
         ^{4\gamma+1}_{L^{8\gamma+2}(\Omega;\mathbb{\R}^d)}
        \Big).
   \end{equation}
   With regard to $\mathbb{I}_6$,
   we first rewrite it as follows
   \begin{equation}
    \begin{split}
       \mathbb{I}_6 
        & =
        \bigg\|
        \E\bigg[
        \int_{-k\tau+jh}^{-k\tau+(j+1)h}
          f
        \big(
          s,
         X_{s}^{-k\tau}
          \big)
          -
          f
         \big(
          jh,
         \Phi
         (X_{-k\tau+jh}^{-k\tau})
          \big)
        \, \dd s 
        \Big|\mathcal{F}_{-k\tau+jh}
        \bigg]
        \bigg\|_{L^2(\Omega;\R^d)}\\
        & =
        \bigg\|
        \E\bigg[
        \int_{-k\tau+jh}^{-k\tau+(j+1)h}
          f
        \big(
          s,
         X_{s}^{-k\tau}
          \big)
          -
          f
         \big(
          jh,
          X_{s}^{-k\tau}
          \big)
        \, \dd s 
        \Big|\mathcal{F}_{-k\tau+jh}
        \bigg] \\
        & \quad +
        \E\bigg[
        \int_{-k\tau+jh}^{-k\tau+(j+1)h}
          f
        \big(
          jh,
         X_{s}^{-k\tau}
          \big)
          -
          f
         \big(
          jh,
         \Phi
         (X_{-k\tau+jh}^{-k\tau})
          \big)
        \, \dd s 
        \Big|\mathcal{F}_{-k\tau+jh}
        \bigg]
        \bigg\|_{L^2(\Omega;\R^d)}.
    \end{split}
   \end{equation}
   Because of the existence of the first derivative of $f$ with respect to the spatial variable,
   then for $s \in [-k\tau+jh,-k\tau+(j+1)h]$,
   \begin{equation}
    \begin{split}
        & f
        \big(
          jh,
         X_{s}^{-k\tau}
          \big)
          -
          f
         \big(
          jh,
         \Phi
         (X_{-k\tau+jh}^{-k\tau})
          \big) \\
          & =
          \int_0^1
          \frac{\partial f}{\partial x}
          \big(
          jh,
          r X_{s}^{-k\tau}
          +(1-r)
          \Phi
         (X_{-k\tau+jh}^{-k\tau})
          \big)
          \big(
          X_{s}^{-k\tau}
          -
          \Phi
         (X_{-k\tau+jh}^{-k\tau})
         \big)
         \,\dd r\\
          & =
          \int_0^1
          \frac{\partial f}{\partial x}
          \big(
          jh,
          r X_{s}^{-k\tau}
          +(1-r)
          \Phi
         (X_{-k\tau+jh}^{-k\tau})
          \big)
          \big(
          X_{s}^{-k\tau}
          -
          X_{-k\tau+jh}^{-k\tau}
          +
          X_{-k\tau+jh}^{-k\tau}
          -
          \Phi
         (X_{-k\tau+jh}^{-k\tau})
         \big)
         \,\dd r\\
         & =
         \underbrace
         {
         \int_0^1
          \frac{\partial f}{\partial x}
          \big(
          jh,
          r X_{s}^{-k\tau}
          +(1-r)
          \Phi
         (X_{-k\tau+jh}^{-k\tau})
          \big)
          \int_{-k\tau+jh}^s
          A  X_l^{-k\tau}
          +
          f
          \big(
          l,X_l^{-k\tau}
          \big)
          \,\dd l
          \,\dd r
          }_{=:\mathbb{J}_1}\\
          & \quad +
          \underbrace
          {
          \int_0^1
          \frac{\partial f}{\partial x}
          \big(
          jh,
          r X_{s}^{-k\tau}
          +(1-r)
          \Phi
         (X_{-k\tau+jh}^{-k\tau})
          \big)
          \int_{-k\tau+jh}^s
          g
          \big(
          l, X_l^{-k\tau}
          \big)
          \,\dd W_l
          \,\dd r
          }_{=:\mathbb{J}_2}\\
          & \quad +
          \underbrace
         {
         \int_0^1
          \frac{\partial f}{\partial x}
          \big(
          jh,
          r X_{s}^{-k\tau}
          +(1-r)
          \Phi
         (X_{-k\tau+jh}^{-k\tau})
          \big)
          \big(
          X_{-k\tau+jh}^{-k\tau}
          \big)
          -
          \Phi
          (
          X_{-k\tau+jh}^{-k\tau}
          )
          \,\dd r
          }_{=:\mathbb{J}_3}
          .\\
    \end{split}
   \end{equation}
   Following a similar argument,
   we get 
   \begin{equation}
       \E\bigg[
       \int_{-k\tau+jh}^{-k\tau+(j+1)h}
         \mathbb{J}_2
       |\mathcal{F}_{-k\tau+jh}
          \,\dd s
          \bigg]
          =0.
          \end{equation}
          Using the Jensen inequality and the H\"older inequality yields
        \begin{equation}
        \begin{split}
        \mathbb{I}_6
        & \leq
        \bigg\|
        \E\bigg[
        \int_{-k\tau+jh}^{-k\tau+(j+1)h}
        \mathbb{J}_1
        \, \dd s 
        \Big|\mathcal{F}_{-k\tau+jh}
        \bigg] \bigg\|_{L^2(\Omega;\R^d)}
        +
        \bigg\|
        \E\bigg[
        \int_{-k\tau+jh}^{-k\tau+(j+1)h}
        \mathbb{J}_3
        \, \dd s 
        \Big|\mathcal{F}_{-k\tau+jh}
        \bigg] \bigg\|_{L^2(\Omega;\R^d)}\\
        & \quad +
        \bigg\|
        \E\bigg[
        \int_{-k\tau+jh}^{-k\tau+(j+1)h}
          f
        \big(
          s,
         X_{s}^{-k\tau}
          \big)
          -
          f
         \big(
          jh,
         X_{s}^{-k\tau}
          \big)
        \, \dd s 
        \Big|\mathcal{F}_{-k\tau+jh}
        \bigg]
        \bigg\|_{L^2(\Omega;\R^d)} \\
        & \leq
        \bigg\|
        \int_{-k\tau+jh}^{-k\tau+(j+1)h}
        \mathbb{J}_1
        \,\dd s
        \bigg\|_{L^2(\Omega;\R^d)}
        +
        \bigg\|
        \int_{-k\tau+jh}^{-k\tau+(j+1)h}
        \mathbb{J}_3
        \,\dd s
        \bigg\|_{L^2(\Omega;\R^d)}\\
        & \quad +
        \bigg\|
        \int_{-k\tau+jh}^{-k\tau+(j+1)h}
          f
        \big(
          s,
         X_{s}^{-k\tau}
          \big)
          -
          f
         \big(
          jh,
         X_{s}^{-k\tau}
          \big)
        \, \dd s 
        \bigg\|_{L^2(\Omega;\R^d)}\\
        & \leq
        \int_{-k\tau+jh}^{-k\tau+(j+1)h}
        \|\mathbb{J}_1\|_{L^2(\Omega;\R^d)}
        \,\dd s 
        +
        \int_{-k\tau+jh}^{-k\tau+(j+1)h}
        \|\mathbb{J}_3\|_{L^2(\Omega;\R^d)}
        \,\dd s\\
        & \quad +
        \int_{-k\tau+jh}^{-k\tau+(j+1)h}
        \|
          f
        \big(
          s,
         X_{s}^{-k\tau}
          \big)
          -
          f
         \big(
          jh,
         X_{s}^{-k\tau})
          \big)
         \|_{L^2(\Omega;\R^d)}
        \, \dd s .
        \end{split}
        \end{equation}
        {\color{blue}}
        In view of the H\"older 
        inequality and 
        \eqref{eq_PMM:the_esti_f'(t,x)},
        \eqref{eq_PMM:the_esti_f(t,x)} 
        and \eqref{eq_PMM:Phi(x)}, 
        \begin{equation}
        \begin{split}
        \|\mathbb{J}_1\|_{L^2(\Omega;\R^d)}
        &\leq
        \int_0^1
        \int_{-k\tau+jh}^s
        \Big\|
        \tfrac{\partial f}{\partial x}
        \big(
        jh,
        r \Phi(X_{s}^{-k\tau})
        +(1-r)
        \Phi
        (X_{-k\tau+jh}^{-k\tau})
        \big)\\
        & \quad \times
        \Big(
        X_l^{-k\tau}
        +
        f\big(
        l,X_l^{-k\tau}
        \big)
        \Big)
        \Big\|
        _{L^2(\Omega;\R^d)}
        \,\dd l
        \,\dd r \\
        & \leq
        \int_0^1
        \int_{-k\tau+jh}^s
        \Big\|
        C\Big(
        1+
        \|
        r \Phi(X_{s}^{-k\tau})
        +(1-r)
        \Phi
        (X_{-k\tau+jh}^{-k\tau})
        \|
        \Big)^{\gamma-1} \\
        & \qquad \times
        \Big(
        1+
        \|
        X_{l}^{-k\tau}
        \|
        \Big)^{\gamma} 
        \Big\|
        _{L^2(\Omega;\R)}
        \,\dd l
        \,\dd r \\
        & \leq
        \int_0^1
        \int_{-k\tau+jh}^s
        \Big\|
        C\Big(
        1+
        \|
        r (X_{s}^{-k\tau}
        +(1-r)
        X_{-k\tau+jh}^{-k\tau}
        \|
        \Big)^{\gamma-1} \\
        & \qquad \times
        \Big(
        1+
        \|
        X_{l}^{-k\tau}
        \|
        \Big)^{\gamma}
        \Big\|
        _{L^2(\Omega;\R)}
        \,\dd l
        \,\dd r \\
        & \leq
        C h
        \Big(
        1+\sup_{k\in \mathbb{N}} \sup_{t \geq -k\tau}
        \big\| 
        X_{t}^{-k\tau} 
        \big\|
        ^{2\gamma-1}_{L^{4\gamma-2}(\Omega;\mathbb{\R}^d)}
        \Big).
        \end{split}
        \end{equation}
        With the aid of the H\"older 
        inequality,
        \eqref{eq_PMM:the_esti_f'(t,x)},
        \eqref{eq_PMM:Phi(x)}
        and 
        Lemma 
        \ref{lem_PMM:the_esti_(X(t1)-X(t2)}, one can show
        \begin{equation}
        \begin{split}
        &
        \|
        \mathbb{J}_3
        \|
        _{L^2(\Omega;\R^d)}
        \\
        &\leq
        \int_0^1
        \Big\|
        \tfrac{\partial f}{\partial x}
        \big(
        jh,
        r \Phi(X_{s}^{-k\tau})
        +(1-r)
        \Phi
        (X_{-k\tau+jh}^{-k\tau})
        \big)
        \big(
        X_{-k\tau+jh}^{-k\tau}
        -
        \Phi
         (
         X_{-k\tau+jh}^{-k\tau}
         )
        \big)
        \Big\|
        _{L^2(\Omega;\R^d)}
        \,\dd r \\
        & \leq
        \int_0^1
        \Big\|
        C\Big(
        1+
        \|
        r \Phi(X_{s}^{-k\tau})
        +(1-r)
        \Phi
        (X_{-k\tau+jh}^{-k\tau})
        \|
        \Big)^{\gamma-1} 
        \times
        Ch^2
        \|
        X_{-k\tau+jh}^{-k\tau}
        \|
        ^{4\gamma+1}
        \Big\|
        _{L^2(\Omega;\R)}
        \,\dd r \\
        & \leq
        \int_0^1 
        C h^2
        \Big\|
        \Big(
        1+
        \|
        r (X_{s}^{-k\tau}
        +(1-r)
        X_{-k\tau+jh}^{-k\tau}
        \|
        \Big)
        \Big\|
        ^{\gamma-1} 
        _{L^{2\kappa_1(\gamma-1)}(\Omega;\R)}
        \times
        \Big\|
        X_{-k\tau+jh}^{-k\tau}
        \Big\|
        ^{4\gamma+1}
        _{L^{2\kappa_2(4\gamma+1)}(\Omega;\R^d)}
        \,\dd r \\
        & \leq
        C h^2
        \Big(
        1+\sup_{k\in \mathbb{N}} \sup_{t \geq -k\tau}
        \big\| 
        X_{t}^{-k\tau} 
        \big\|
        ^{5\gamma}
        _{L^{10\gamma}(\Omega;\R^d)}
        \Big),
        \end{split}
        \end{equation}
        where we take
        $\frac{1}{\kappa_1}=
        \frac{5\gamma}{4\gamma+1}$,
        $\frac{1}{\kappa_2}=
        \frac{5\gamma}{\gamma-1}$
        such that $\frac{1}{\kappa_1}+
        \frac{1}{\kappa_2}=1$.
        Owing to \eqref{eq_PMM:the_esti_f(t,x)-f(s,x)},
        one easily gets
        \begin{equation}
        \begin{split}
        \|
          f
        \big(
          s,
         X_{s}^{-k\tau}
          \big)
          -
          f
         \big(
          jh,
         X_{s}^{-k\tau}
          \big)
         \|_{L^2(\Omega;\R^d)} 
         \leq
         C h
        \Big(
        1+\sup_{k\in \mathbb{N}} \sup_{t \geq -k\tau}
        \big\| 
        X_{t}^{-k\tau} 
        \big\|
        ^{\gamma}_{L^{2\gamma}(\Omega;\mathbb{\R}^d)}
        \Big).
        \end{split}
        \end{equation}
        Hence,
        \begin{equation}
            \mathbb{I}_6
            \leq
            C h^2
        \Big(
        1+\sup_{k\in \mathbb{N}} \sup_{t \geq -k\tau}
        \big\| 
        X_{t}^{-k\tau} 
        \big\|
        ^{5\gamma}_{L^{10\gamma}(\Omega;\mathbb{\R}^d)}
        \Big).
        \end{equation}
    Thanks to \eqref{lem_PMM:x-Phi(x)},
    one can easily get
        \begin{equation}
            \mathbb{I}_7
            \leq
            Ch^2
            \|
       X_{-k\tau+jh}
       ^{-k\tau}
       \|
       ^{4\gamma+1}
       _{L^{8\gamma+2}(\Omega;\R^d)}.
        \end{equation}
        Therefore,
        from
        \eqref{eq_PMM:R_j+1|F}
        it immediately follows
        that
        \begin{equation}
        \big\|
        \E[
        \mathcal{R}_{-k\tau+(j+1)h}
        | \mathcal{F}
        _{-k\tau+jh}
        ]\big\|_{L^2(\Omega;\R^d)}
        \leq
        C h^2
        \Big(
        1+\sup_{k\in \mathbb{N}} \sup_{t \geq -k\tau}
        \big\| 
        X_{t}^{-k\tau} 
        \big\|
        ^{5\gamma}
        _{L^{10\gamma}(\Omega;\mathbb{\R}^d)}
        \Big).
        \end{equation}
    This thus finishes the proof of the lemma.
\end{proof}

\subsection{The order-one convergence of PMM}
We are now prepared to present the main result of this section,
which demonstrates the order-one
convergence of the explicit Milstein scheme for the SDE \eqref{eq_PMM:Problem_SDE}
in the infinite time horizon.
\begin{thm}
    \label{thm_PMM:error analysis}
    Let Assumptions \ref{ass_PMM}, \ref{ass_PMM：polynomial_growth} be fulfilled with $\gamma \in \big[1,\frac{p^*}{5}\big]$.
    Let  $X_{-k\tau+jh}^{-k\tau}$ and 
    $\tilde{X}_{-k\tau+jh}^{-k\tau}$ 
    be given by \eqref{eq_PMM:Problem_SDE} and \eqref{eq:the_projected_milstein_method}, respectively.
    For an arbitrary pair 
    $(\sigma_1,\sigma_2)$
    satisfying $\sigma_1
    \in (0,\lambda_1
    -K_2
    -2\beta_{\mathcal{L}}^2)$
    and $\sigma_2>0$,
    there exists a positive constant $C$,
    independent of $k,j 
    \in \N$,
    such that
    \begin{equation}
        \sup_{k,j\in N}
        \E
        [\|
        X_{-k\tau+jh}^{-k\tau}
        -
        \tilde{X}_{-k\tau+jh}^{-k\tau}
        \|^2]
        \leq
        Ch^2,
    \end{equation}
    where the timestep $h$
    satisfies
    \begin{equation}
         h
        \in
    \bigg(
    0,
    \min
    \bigg\{
    \dfrac{(\lambda_1-K_2)^{\gamma}}{
    (1+\sigma_2)^{\gamma}(\lambda_d+\beta_f)^{2\gamma}},
    \dfrac{1}
    {
    \lambda_1
    -K_2
    -\sigma_1
    -2\beta_{\mathcal{L}}^2
    }
    ,
    1
    \bigg\}
    \bigg).
    \end{equation}
\end{thm}
\begin{proof}
    [Proof of Theorem \ref{thm_PMM:error analysis}
    ]
    For brevity,
    for $j,k \in \N$
    we denote
    \begin{equation}
    \label{eq_PMM:the_short-hand_notation}
        \begin{split}
            e_{-k\tau+jh}
            &:=
            X_{-k\tau+jh}^{-k\tau}
            -
            \tilde{X}_{-k\tau+jh}^{-k\tau},\\
            \Delta X^{\Phi}_{-k\tau+jh}
            &:=
            \Phi
            \big(
            X_{-k\tau+jh}^{-k\tau}
            \big)
            -
            \Phi
            \big(
            \tilde{X}_{-k\tau+jh}^{-k\tau}
            \big),\\
            \Delta f^{\Phi}_{-k\tau+jh}
            &:=
            f
            \Big(jh,
            \Phi
            \big(
            X_{-k\tau+jh}^{-k\tau}
            \big)
            \Big)
            -
            f
            \Big(
            jh,
            \Phi
            \big(
            \tilde{X}_{-k\tau+jh}^{-k\tau}
            \big)
            \Big),\\
            \Delta g^{\Phi}_{-k\tau+jh}
            &:=
            g
            \Big(jh,
            \Phi
            \big(
            X_{-k\tau+jh}^{-k\tau}
            \big)
            \Big)
            -
            g
            \Big(
            jh,
            \Phi
            \big(
            \tilde{X}_{-k\tau+jh}^{-k\tau}
            \big)
            \Big),\\
            \Delta \big( \mathcal{L}^{r_1}g_{r_2}\big)
            ^{\Phi}_{-k\tau+jh}
            &:=
            \mathcal{L}^{r_1}g_{r_2}
            \Big(jh,
            \Phi
            \big(
            X_{-k\tau+jh}^{-k\tau}
            \big)
            \Big)
            -
            \mathcal{L}^{r_1}g_{r_2}
            \Big(
            jh,
            \Phi
            \big(
            \tilde{X}_{-k\tau+jh}^{-k\tau}
            \big)
            \Big).\\
        \end{split}
    \end{equation}
    Using the short-hand notation \eqref{eq_PMM:the_short-hand_notation} and
    subtracting \eqref{eq:the_projected_milstein_method}
     from \eqref{eq_PMM:the_exact_X_j+1} gives
     \begin{equation}
     \label{eq_PMM:e_j+1}
        \begin{split}
        e_{-k\tau+(j+1)h}
        & =
        \Delta X^{\Phi}_{-k\tau+jh}
        +
        h A
        \Delta X^{\Phi}_{-k\tau+jh}
        +
        h
        \Delta f^{\Phi}_{-k\tau+jh}
        +
        \Delta g^{\Phi}_{-k\tau+jh}
        \Delta W_{-k\tau+jh}\\
        & \quad +
        \sum_{r_1,r_2=1}^{m}
        \Delta 
        \big( \mathcal{L}^{r_1}g_{r_2}
        \big)
        ^{\Phi}_{-k\tau+jh}
        \Pi^{-k\tau+jh,-k\tau+(j+1)h}_{r_1,r_2}
        +
        \mathcal{R}_{-k\tau+(j+1)h}.
        \end{split}
     \end{equation}
     Squaring both sides of \eqref{eq_PMM:e_j+1}
     yields
     \begin{equation}
     \label{eq_PMM:|e_j+1|^2}
        \begin{split}
        \|e_{-k\tau+(j+1)h}\|^2
        & =
        \|
        \Delta X^{\Phi}_{-k\tau+jh}
        \|^2
        +
        h^2
        \|
        A 
        \Delta X^{\Phi}_{-k\tau+jh}
        +
        \Delta f^{\Phi}_{-k\tau+jh}
        \|^2
        +
        \|
        \Delta g^{\Phi}_{-k\tau+jh}
        \Delta W_{-k\tau+jh}
        \|^2 \\
        & \quad +
        \Big\|
        \sum_{r_1,r_2=1}^{m}
        \Delta 
        \big( \mathcal{L}^{r_1}g_{r_2}
        \big)
        ^{\Phi}_{-k\tau+jh}
        \Pi^{-k\tau+jh,-k\tau+(j+1)h}_{r_1,r_2}
        \Big\|^2
        +
        \|
        \mathcal{R}_{-k\tau+(j+1)h}
        \|^2 \\
        & \quad +
        2h
        \big\langle
        \Delta X^{\Phi}_{-k\tau+jh}
        ,
        A 
        \Delta X^{\Phi}_{-k\tau+jh}
        +
        \Delta f^{\Phi}_{-k\tau+jh}
        \big\rangle
        +
        2h
        \big\langle
        \Delta X^{\Phi}_{-k\tau+jh}
        ,
        \Delta g^{\Phi}_{-k\tau+jh}
        \Delta W_{-k\tau+jh}
        \big\rangle \\
        & \quad +
        2 \bigg\langle
        \Delta X^{\Phi}_{-k\tau+jh}
        ,
        \sum_{r_1,r_2=1}^{m}
        \Delta 
        \big( \mathcal{L}^{r_1}g_{r_2}
        \big)
        ^{\Phi}_{-k\tau+jh}
        \Pi^{-k\tau+jh,-k\tau+(j+1)h}_{r_1,r_2}
        \bigg\rangle
        +
        2 \big\langle
        \Delta X^{\Phi}_{-k\tau+jh}
        ,
        \mathcal{R}_{-k\tau+(j+1)h}
        \big\rangle \\
        & \quad +
        2h \big\langle
        A 
        \Delta X^{\Phi}_{-k\tau+jh}
        +
        \Delta f^{\Phi}_{-k\tau+jh}
        ,
        \Delta g^{\Phi}_{-k\tau+jh}
        \Delta W_{-k\tau+jh}
        \big\rangle \\
        & \quad +
        2h
        \Big\langle
        A 
        \Delta X^{\Phi}_{-k\tau+jh}
        +
        \Delta f^{\Phi}_{-k\tau+jh}
        ,
        \sum_{r_1,r_2=1}^{m}
        \Delta 
        \big( \mathcal{L}^{r_1}g_{r_2}
        \big)
        ^{\Phi}_{-k\tau+jh}
        \Pi^{-k\tau+jh,-k\tau+(j+1)h}_{r_1,r_2}
        \Big\rangle \\
        & \quad +
        2h
        \big\langle
        A 
        \Delta X^{\Phi}_{-k\tau+jh}
        +
        \Delta f^{\Phi}_{-k\tau+jh}
        ,
        \mathcal{R}_{-k\tau+(j+1)h}
        \big\rangle \\
        & \quad +
        2 \bigg\langle
        \Delta g^{\Phi}_{-k\tau+jh}
        \Delta W_{-k\tau+jh}
        ,
        \sum_{r_1,r_2=1}^{m}
        \Delta 
        \big( \mathcal{L}^{r_1}g_{r_2}
        \big)
        ^{\Phi}_{-k\tau+jh}
        \Pi^{-k\tau+jh,-k\tau+(j+1)h}_{r_1,r_2}
        \bigg\rangle \\
        & \quad +
        2 \big\langle
        \Delta g^{\Phi}_{-k\tau+jh}
        \Delta W_{-k\tau+jh}
        ,
        \mathcal{R}_{-k\tau+(j+1)h}
        \big\rangle \\
        & \quad +
        2 \bigg\langle
        \sum_{r_1,r_2=1}^{m}
        \Delta 
        \big( \mathcal{L}^{r_1}g_{r_2}
        \big)
        ^{\Phi}_{-k\tau+jh},
        \Pi^{-k\tau+jh,-k\tau+(j+1)h}_{r_1,r_2},
        \mathcal{R}_{-k\tau+(j+1)h}
        \bigg\rangle. \\
        \end{split}
     \end{equation}
     It is easy to see
     \begin{equation}
         \begin{split}
         \E
         \big[
         \|
         \Delta g^{\Phi}_{-k\tau+jh}
         \Delta W_{-k\tau+jh}
         \|^2
         \big]
         =
          h
         \E
         \big[
         \|
         \Delta g^{\Phi}_{-k\tau+jh}
         \|^2
         \big].
\end{split}
\end{equation}
Also, we claim
\begin{equation}
\begin{split}
         \E
         \Big[
         \Big\|
        \sum_{r_1,r_2=1}^{m}
        \Delta 
        \big( \mathcal{L}^{r_1}g_{r_2}
        \big)
        ^{\Phi}_{-k\tau+jh}
        \Pi^{-k\tau+jh,-k\tau+(j+1)h}_{r_1,r_2}
        \Big\|^2
         \Big]
        \leq
         h^2
         \sum_{r_1,r_2=1}^{m}
         \E
         \Big[
         \big\|
         \Delta 
         \big( \mathcal{L}^{r_1}g_{r_2}
         \big)
         ^{\Phi}_{-k\tau+jh}
         \big\|^2
         \Big].
         \end{split}
     \end{equation}
Indeed, in the case $\{r_1,r_2\} \neq 
     \{r_3,r_4\}$,
     one can infer that
     \begin{equation}
         \E
         \Big[
         \Big\langle
         \sum_{r_1,r_2=1}^m
         \Delta 
        \Big( \mathcal{L}^{r_1}g_{r_2}
        \big)
        ^{\Phi}_{-k\tau+jh}
        \Pi^{-k\tau+jh,-k\tau+(j+1)h}_{r_1,r_2}
        ,
        \sum_{r_3,r_4=1}^m
         \Delta 
        \Big( \mathcal{L}^{r_3}g_{r_4}
        \big)
        ^{\Phi}_{-k\tau+jh}
        \Pi^{-k\tau+jh,-k\tau+(j+1)h}_{r_3,r_4}
         \Big\rangle
         \Big]
         =
         0.
     \end{equation}
 As a consequence,
     \begin{equation}
     \begin{split}
         &\E
         \Big[
         \Big\|
        \sum_{r_1,r_2=1}^{m}
        \Delta 
        \big( \mathcal{L}^{r_1}g_{r_2}
        \big)
        ^{\Phi}_{-k\tau+jh}
        \Pi^{-k\tau+jh,-k\tau+(j+1)h}_{r_1,r_2}
        \Big\|^2
        \Big]
        \\
        & \quad =
        \sum_{r_1,r_2=1}^m
        \E
         \Big[
         \Big\|
        \Delta 
        \big( \mathcal{L}^{r_1}g_{r_2}
        \big)
        ^{\Phi}_{-k\tau+jh}
        \Pi^{-k\tau+jh,-k\tau+(j+1)h}_{r_1,r_2}
        \Big\|^2
        \Big] \\
        & \qquad +
        \sum
        _{r_1,r_2=1,r_1 \neq r_2}^m
        \E
         \Big[
         \Big
         \langle
        \Delta 
        \big( \mathcal{L}^{r_1}g_{r_2}
        \big)
        ^{\Phi}_{-k\tau+jh}
        \Pi^{-k\tau+jh,-k\tau+(j+1)h}_{r_1,r_2}
        ,
        \Delta 
        \big( \mathcal{L}^{r_2}g_{r_1}
        \big)
        ^{\Phi}_{-k\tau+jh}
        \Pi^{-k\tau+jh,-k\tau+(j+1)h}_{r_2,r_1}
        \Big
        \rangle
        \Big] \\
        & \quad \leq
        \sum_{r_1,r_2=1}^m
        \E
         \Big[
         \Big\|
        \Delta 
        \big( \mathcal{L}^{r_1}g_{r_2}
        \big)
        ^{\Phi}_{-k\tau+jh}
        \Pi^{-k\tau+jh,-k\tau+(j+1)h}_{r_1,r_2}
        \Big\|^2
        \Big] \\
        & \qquad +
        \tfrac{1}{2}
        \sum_{r_1,r_2=1}^m
        \E
         \Big[
         \Big\|
        \Delta 
        \big( \mathcal{L}^{r_1}g_{r_2}
        \big)
        ^{\Phi}_{-k\tau+jh}
        \Pi^{-k\tau+jh,-k\tau+(j+1)h}_{r_1,r_2}
        \Big\|^2
        \Big] \\
        & \qquad +
        \tfrac{1}{2}
        \sum_{r_2,r_1=1}^m
        \E
         \Big[
         \Big\|
        \Delta 
        \big( \mathcal{L}^{r_2}g_{r_1}
        \big)
        ^{\Phi}_{-k\tau+jh}
        \Pi^{-k\tau+jh,-k\tau+(j+1)h}_{r_2,r_1}
        \Big\|^2
        \Big] \\
        & \quad \leq
          h^2
         \sum_{r_1,r_2=1}^{m}
         \E
         \Big[
         \big\|
         \Delta 
         \big( \mathcal{L}^{r_1}g_{r_2}
         \big)
         ^{\Phi}_{-k\tau+jh}
         \big\|^2
         \Big],
    \end{split}
    \end{equation}
validating the above claim.
    Furthermore,
     \begin{equation}
         \begin{split}
         \E
         \Big[
         \big\langle
        \Delta X^{\Phi}_{-k\tau+jh}
        ,
        \Delta g^{\Phi}_{-k\tau+jh}
        \Delta W_{-k\tau+jh}
        \big\rangle
         \Big]
         & =0,\\
         \E
         \Big[
         \Big\langle
        \Delta X^{\Phi}_{-k\tau+jh}
        ,
        \sum_{r_1,r_2=1}^{m}
        \Delta 
        \big( \mathcal{L}^{r_1}g_{r_2}
        \big)
        ^{\Phi}_{-k\tau+jh}
        \Pi^{-k\tau+jh,-k\tau+(j+1)h}_{r_1,r_2}
        \Big\rangle
         \Big]
         & =0,\\
         \E
         \Big[
         \big\langle
         A 
        \Delta X^{\Phi}_{-k\tau+jh}
        +
        \Delta f^{\Phi}_{-k\tau+jh}
        ,
        \Delta g^{\Phi}_{-k\tau+jh}
        \Delta W_{-k\tau+jh}
        \big\rangle 
         \Big]
         & =0,\\
         \E
         \Big[
         \Big\langle
        A 
        \Delta X^{\Phi}_{-k\tau+jh}
        +
        \Delta f^{\Phi}_{-k\tau+jh}
        ,
        \sum_{r_1,r_2=1}^{m}
        \Delta 
        \big( \mathcal{L}^{r_1}g_{r_2}
        \big)
        ^{\Phi}_{-k\tau+jh}
        \Pi^{-k\tau+jh,-k\tau+(j+1)h}_{r_1,r_2}
        \Big\rangle
         \Big]
         & =0,\\
         \E
         \Big[
         \Big\langle
        \Delta g^{\Phi}_{-k\tau+jh}
        \Delta W_{-k\tau+jh}
        ,
        \sum_{r_1,r_2=1}^{m}
        \Delta 
        \big( \mathcal{L}^{r_1}g_{r_2}
        \big)
        ^{\Phi}_{-k\tau+jh}
        \Pi^{-k\tau+jh,-k\tau+(j+1)h}_{r_1,r_2}
        \Big\rangle
         \Big]
         & =0,
         \end{split}
     \end{equation}
     where the fact was used that the terms $\Delta X^{\Phi}_{-k\tau+jh}$,
     $\Delta f^{\Phi}_{-k\tau+jh}$, 
     $\Delta g^{\Phi}_{-k\tau+jh}$ and 
     $\Delta 
        \big( \mathcal{L}^{r_1}g_{r_2}
        \big)
        ^{\Phi}_{-k\tau+jh}$
     is $\mathcal{F}_{-k\tau+jh}$-measurable.\\
    Equipped with these estimates and taking expectations on both sides of \eqref{eq_PMM:|e_j+1|^2},
    one can then derive
    \begin{equation}
    \begin{split}
        \E
        [
        \|e_{-k\tau+(j+1)h}
        \|^2
        ]
        &\leq
        \E[
        \|\Delta X^{\Phi}_{-k\tau+jh}\|^2
        ]
        +
        h^2
        \E
        [\|A\Delta X^{\Phi}_{-k\tau+jh}
        +
        \Delta f^{\Phi}_{-k\tau+jh}
        \|^2
        ]
        +
        h
        \E
        [\|\Delta g^{\Phi}_{-k\tau+jh}\|^2] \\
        & \quad +
        \tfrac{h^2}{2}
        \sum_{r_1,r_2=1}^{m}
        \E
        \Big[
        \big\|
        \Delta 
        (\mathcal{L}^{r_1}g_{r_2})
        ^{\Phi}_{-k\tau+jh}
        \big\|
        \Big]
        +
        \E
        [\|
        \mathcal{R}_{-k\tau+(j+1)h}
        \|^2] \\
        & \quad +
        2h
        \E
        \Big[
        \big\langle
        \Delta X^{\Phi}_{-k\tau+jh},
        A\Delta X^{\Phi}_{-k\tau+jh}
        +
        \Delta f^{\Phi}_{-k\tau+jh}
        \big\rangle
        \Big]
        +
        2\E
        \Big[
        \big\langle
        \Delta X^{\Phi}_{-k\tau+jh},
        \mathcal{R}_{-k\tau+(j+1)h}
        \big\rangle
        \Big] \\
        & \quad +
        2h\E
        \Big[
        \big\langle
        A\Delta X^{\Phi}_{-k\tau+jh}
        +
        \Delta f^{\Phi}_{-k\tau+jh},
        \mathcal{R}_{-k\tau+(j+1)h}
        \big\rangle
        \Big]\\
        & \quad +
        2 \E
        \Big[
        \big\langle
        \Delta g^{\Phi}_{-k\tau+jh}
        \Delta W_{-k\tau+jh},
        \mathcal{R}_{-k\tau+(j+1)h}
        \big\rangle
        \Big] \\
        & \quad +
        2 \E
        \Big[
        \Big\langle
         \sum_{r_1,r_2=1}^{m}
        \Delta 
        \big( \mathcal{L}^{r_1}g_{r_2}
        \big)
        ^{\Phi}_{-k\tau+jh}
        \Pi^{-k\tau+jh,-k\tau+(j+1)h}_{r_1,r_2},
        \mathcal{R}_{-k\tau+(j+1)h}
        \Big\rangle
        \Big] .
    \end{split}
    \end{equation}
    Recalling $\Delta X^{\Phi}_{-k\tau+jh}$ is $\mathcal{F}_{-k\tau+jh}$-measurable and
    using the Cauchy-Schwartz inequality $2ab\leq\sigma_1ha^2+\frac{1}{\sigma_1h}b^2$ with
    $\sigma_1 \in (0,\lambda_1-K_2
    -2\beta_{\mathcal{L}}^2
    )$,
    one can get
    \begin{equation}
    \begin{split}
        2\E
        \Big[
        \big\langle
        \Delta X^{\Phi}_{-k\tau+jh},
        \mathcal{R}_{-k\tau+(j+1)h}
        \big\rangle
        \Big]
        & =
        2
        \E
        \Big[\E
        \big\langle
        \Delta X^{\Phi}_{-k\tau+(j+1)h},
        \mathcal{R}
        _{-k\tau+(j+1)h}
        \big\rangle
        \big| \mathcal{F}_{-k\tau+jh}
        \Big] \\
        & =
        2
        \E
        \Big[
        \big\langle
        \Delta X^{\Phi}_{-k\tau+jh},
        \E\big[
        \mathcal{R}
        _{-k\tau+(j+1)h}
        \big| \mathcal{F}_{-k\tau+jh}
        \big]
        \big\rangle
        \Big] \\
        & \leq
        \sigma_1 h
        \E[\|
        \Delta X^{\Phi}_{-k\tau+jh}
        \|^2]
        +
        \tfrac{1}{\sigma_1 h}
        \E\big[
        \E\big[
        \|
        \mathcal{R}
        _{-k\tau+(j+1)h}
        \big| \mathcal{F}_{-k\tau+jh}
        \|^2
        \big]
        \big].
    \end{split}
    \end{equation}
    Likewise, for a positive $\sigma_2$,
    using the Young inequality gives,
    \begin{equation}
    \begin{split}
        &
        2h\E
        \Big[
        \big\langle
        A\Delta X^{\Phi}_{-k\tau+jh}
        +
        \Delta f^{\Phi}_{-k\tau+jh},
        \mathcal{R}
        _{-k\tau+(j+1)h}
        \big\rangle
        \Big] \\
        & 
        \quad \leq
        \sigma_2 
        h^2 
        \E
        [\|
        A\Delta X^{\Phi}_{-k\tau+jh}
        +
        \Delta 
        f^{\Phi}_{-k\tau+jh}
        \|^2]
        +
        \tfrac{1}{\sigma_2}
        \E
        [
        \| \mathcal{R}
        _{-k\tau+(j+1)h} 
        \|^2
        ].
    \end{split}
    \end{equation}
    Similarly, for $q \geq 1 $ 
    coming from Assumption \ref{ass_PMM:generalizedz_monotonicity_condition},
    \begin{equation}
    \begin{split}
        2\E
        \Big[
        \big\langle
        \Delta g^{\Phi}
        _{-k\tau+jh}
        \Delta W
        _{-k\tau+jh},
        \mathcal{R}
        _{-k\tau+(j+1)h}
        \big\rangle
        \Big] 
        \leq
        (2q-2)h
        \E
        [\|
        \Delta 
        g^{\Phi}
        _{-k\tau+jh}\|^2]
        +
        \tfrac{1}{2q-2}
        \E
        [
        \| 
        \mathcal{R}
        _{-k\tau+(j+1)h} 
        \|
        ^2],
    \end{split}
    \end{equation}
    \begin{equation}
    \begin{split}
        &2 \E
        \Big[
        \Big\langle
         \sum_{r_1,r_2=1}^{m}
        \Delta 
        \big( \mathcal{L}^{r_1}g_{r_2}
        \big)
        ^{\Phi}
        _{-k\tau+jh}
        \Pi^{-k\tau+jh,-k\tau+(j+)h}
        _{r_1,r_2},
        \mathcal{R}
        _{-k\tau+(j+1)h}
        \Big\rangle
        \Big] \\
        & \quad \leq
        h^2
        \sum_{r_1,r_2=1}^{m}
        \E\Big[
        \big\|
        \Delta 
        \big( \mathcal{L}^{r_1}g_{r_2}
        \big)
        ^{\Phi}
        _{-k\tau+jh}
        \big\|^2
        \Big]
        +
        \E
        [\| \mathcal{R}
        _{-k\tau+(j+1)h} \|^2].
    \end{split}
    \end{equation}
    Taking these estimates into consideration, 
    one immediately arrives at
    \begin{equation}
    \label{eq_PMM:the_esti_E[|e_j+1|^2]}
    \begin{split}
        \E
        [
        \|e_{-k\tau+(j+1)h}
        \|^2
        ]
        &\leq
        (1+\sigma_1 h)
        \E[
        \|\Delta X^{\Phi}_{-k\tau+jh}\|^2
        ]
        +
        (1+\sigma_2)
        h^2
        \E
        [\|A\Delta X^{\Phi}_{-k\tau+jh} +
        \Delta f^{\Phi}_{-k\tau+jh}
        \|^2
        ]\\
        & \quad +
        2h
        \E
        \Big[
        \big\langle
        \Delta X^{\Phi}_{-k\tau+jh},
        A\Delta X^{\Phi}_{-k\tau+jh}
        +
        \Delta f^{\Phi}_{-k\tau+jh}
        \big\rangle
        \Big] \\
        & \quad +
        (2q-1)h
        \E
        [\|\Delta g^{\Phi}_{-k\tau+jh}\|^2]
        +
        2
        h^2
        \sum_{r_1,r_2=1}^{m}
        \E
        \Big[
        \big\|
        \Delta 
        (\mathcal{L}^{r_1}g_{r_2})
        ^{\Phi}_{-k\tau+jh}
        \big\|^2
        \Big] \\
        & \quad +
        (2
        +\tfrac{1}{\sigma_2}
        +\tfrac{1}{2p-2}
        )
        \E
        [\|
        \mathcal{R}
        _{-k\tau+(j+1)h}
        \|^2
        ] \\
        & \quad +
        \tfrac{1}{\sigma_1 h}
        \E\big[
        \E\big[
        \|
        \mathcal{R}_{-k\tau+jh}
        \big| 
        \mathcal{F}
        _{-k\tau+(j+1)h}
        \|^2
        \big]
        \big].
    \end{split}
    \end{equation}
    Additionally, applying Lemma \ref{lem_PMM:Phi(x)} gives
    \begin{equation}
    \label{eq_PMM:E[AX+f]}
    \begin{split}
        &h^2
        \E
        [\|A\Delta X^{\Phi}_{-k\tau+jh} +
        \Delta f^{\Phi}_{-k\tau+jh}
        \|^2
        ] \\
        & \leq
        \lambda_d^2 h^2
        \E[\|
        \Delta X^{\Phi}_{-k\tau+jh}\|^2]
        +2\lambda_d \beta_f
        h^{1+\frac{1}{2\gamma}}
        \E[\|
        \Delta X^{\Phi}_{-k\tau+jh}\|^2]
        +
        \beta_f^2
        h^{1+\frac{1}{\gamma}}
        \E[\|
        \Delta X^{\Phi}_{-k\tau+jh}\|^2] \\
        & \leq
        \Big(
        (\lambda_d
        +\beta_f)^2
        h^{\frac{1}{\gamma}}
        \Big)h
        \E[\|
        \Delta X^{\Phi}_{-k\tau+jh}\|^2],
    \end{split}
    \end{equation}
    \begin{equation}
    \label{eq_PMM:the_esti_h^2E|Delta_L|^2}
        2h^2
        \sum_{r_1,r_2=1}^{m}
        \E
        \Big[
        \big\|
        \Delta 
        (\mathcal{L}^{r_1}g_{r_2})
        ^{\Phi}_{-k\tau+jh}
        \big\|^2
        \Big] 
        \leq
        2
        \beta_{\mathcal{L}}^2
        h^{1+\frac{1}{\gamma}}
        \E
        [
        \|
        \Delta X^{\Phi}_{-k\tau+jh}
        \|^2
        ]
        \leq
        2\beta_{\mathcal{L}}^2
        h
        \E
        [
        \|
        \Delta X^{\Phi}_{-k\tau+jh}
        \|^2
        ].
    \end{equation}
    Here we select an appropriate $h$ such that
    \begin{equation}
         h
        \in
    \Big(
    0,
    \min
    \Big\{
    \tfrac{(\lambda_1-K_2)^{\gamma}}{(1+\sigma_2)^{\gamma}(\lambda_d+\beta_f)^{2\gamma}},
    \tfrac{1}
    {
    \lambda_1
    -K_2
    -\sigma_1
    -
    2\beta_{\mathcal{L}}^2
    }
    ,
    1
    \Big\}
    \Big)
    \end{equation}
    to ensure
    \begin{equation}
        (1+\sigma_2)
        (\lambda_d
        +\beta_f)^2
        h^{\frac{1}{\gamma}}
        <
        \lambda_1-K_2,
        \quad
        1-
        (
        \lambda_1
        -K_2
        -\sigma_1
        -
        2\beta_{\mathcal{L}}^2
        )
        h>0.
    \end{equation}
    Inserting this into \eqref{eq_PMM:the_esti_E[|e_j+1|^2]} and recalling Assumption \ref{ass_PMM} and \ref{ass_PMM:generalizedz_monotonicity_condition} yield
    \begin{equation}
    \begin{split}
        \E
        [
        \|e_{-k\tau+(j+1)h}
        \|^2
        ]
        & =
        \big\{
        1-
        (
        \lambda_1
        -K_2
        -\sigma_1
        -
        2\beta_{\mathcal{L}}^2
        )
        h
        \big\}
        \E
        [
        \|e_{-k\tau+jh}
        \|^2
        ]\\
        & \quad +
        \big(
        2
        +\tfrac{1}{\sigma_2}
        +\tfrac{1}{2p-2}
        \big)
        \E
        [\|
        \mathcal{R}
        _{-k\tau+(j+1)h}
        \|^2] 
        +
        \tfrac{1}{\sigma_1 h}
        \E\big[
        \E\big[
        \|
        \mathcal{R}
        _{-k\tau+(j+1)h}
        \big| 
        \mathcal{F}
        _{-k\tau+jh}
        \|^2
        \big]
        \big].\\
    \end{split}
    \end{equation}
    Denoting 
    $C_A:=
    \lambda_1-K_2
    -\sigma_1
    -
    2\beta_{\mathcal{L}}^2$,
    and recalling Lemma \ref{lem_PMM:the_esti_R_j+1}, we have
    \begin{equation}
    \begin{split}
        \E
        [
        \|e_{-k\tau+(j+1)h}
        \|^2
        ]
        &\leq
        (1-C_A h)
        \E
        [
        \|e_{-k\tau+jh}
        \|^2
        ]
        +Ch^3 \\
        & \leq
        (1-C_A h)^{j+1}
        \E
        [
        \|e_{-k\tau}
        \|^2
        ]
        +
        Ch^3
        \sum_{i=0}^j
        (1-C_A h)^i\\
        & =
        (1-C_A h)^{j+1}
        \E
        [
        \|e_{-k\tau}
        \|^2
        ]
        +
        Ch^3 \times
        \tfrac{1-(1-C_A h)^{j}}{C_A h}.
    \end{split}
    \end{equation}
    By the observation of
    $e_{-k\tau}=0$, we finally obtain
    \begin{equation}
        \E
        [
        \|e_{-k\tau+(j+1)h}
        \|^2
        ]
        \leq
        Ch^2,
    \end{equation}
     as required.
\end{proof}

\section{The Random Periodic Solution of the projected Milstein scheme}
In this section,
we focus on the  the existence and uniqueness of the random periodic solution of numerical solutions.
The next corollary confirms a uniform bound for the second moment of the numerical solution.
\begin{cor}
    \label{cor_PMM:second moment of PMM}
    Let Assumption \ref{ass_PMM}
    hold and let $
    \big\{
    {\Tilde{X}}
    _{-k\tau+(j+1)h}
    ^{-k\tau}
    \big\}_{k,j \in \N}$
    be given by \eqref{eq:the_projected_milstein_method}.
    Then 
    \begin{equation}
   \label{eq_PMM:second moment of PEM}
        \sup_{k,j \in \mathbb{N}}
        \mathbb{E}
        \Big[
        \big\|
        {\Tilde{X}}
        _{-k\tau+(j+1)h}
        ^{-k\tau}
        \big\|
        ^{2}
        \Big]
        <
        \infty.
   \end{equation}
\end{cor}
\begin{proof}
[Proof of Corollary \ref{cor_PMM:second moment of PMM}]
    Combining Theorem
    \ref{thm_PMM:error analysis} with Lemma
    \ref{lem_PMM:the_pth_of_exact_solution}
    yields
    \begin{equation}
        \sup_{k,j\in N}
        \E
        [\|
        \tilde{X}_{-k\tau+jh}^{-k\tau}
        \|^2] 
        \leq
        2\sup_{k,j\in N}
        \E
        [\|
        \tilde{X}_{-k\tau+jh}^{-k\tau}
        -
        X_{-k\tau+jh}^{-k\tau}
        \|^2]
        +
        2
        \sup_{k,j\in N}
        \E
        [\|
        X_{-k\tau+jh}^{-k\tau}
        \|^2] 
        <
        \infty.
    \end{equation}
\end{proof}
The following lemma indicates that any two numerical solutions,
starting from different initial conditions,
can become arbitrarily close after a sufficiently large number of iterations.
\begin{lem}  \label{lem_PMM:contractivity_of_PMM}
  Let
    Assumption \ref{ass_PMM} 
    hold.
    Let
    ${ \tilde{X}^{-k\tau}_{-k\tau+jh}}$ and ${ \tilde{Y}^{-k\tau}_{-k\tau+jh}}$ 
    be two solutions of the 
    projected Milstein scheme
   \eqref{eq:the_projected_milstein_method}
    with initial values $\xi$
    and $\eta$ satisfying condition (iv) in Assumption \ref{ass_PMM}, respectively.
    Then it holds that
    \begin{equation}
        \label{eq_PMM:the_seond_moment_|x-y|}
        \E
        \Big[
        \big\|
        \tilde{X}
        ^{-k\tau}_{-k\tau+(j+1)h}
        -
        \tilde{Y}
        ^{-k\tau}_{-k\tau+(j+1)h}
        \big\|
        ^{2}
        \Big]
        \leq
        e
        ^
        {
        -
        (
        \lambda_1
        -K_2
        -
        \beta_{\mathcal{L}}^2}
        )
        (j+1)h
        \E
        [
        \| \xi-\eta \|^{2}  ],
    \end{equation}
    where $h$ is the timestep satisfying
    \begin{equation}
         h
        \in
    \bigg(
    0,
    \min
    \bigg\{
    \dfrac
    {
    (\lambda_1
        -K_2
        -
        \beta_{\mathcal{L}}^2)^{\gamma}}{{(\lambda_d+\beta_f)^{2\gamma}}},
    \dfrac{1}{\lambda_1-K_2-
\beta_{\mathcal{L}}^2}
    ,
    1
    \bigg\}
    \bigg).
    \end{equation}
\end{lem}
\begin{proof}
    [Proof of Lemma \ref{lem_PMM:contractivity_of_PMM}]
    To simplify the notation,
    we denote
    \begin{equation}
        \begin{split}
        &\tilde{Z}_j:=
        \tilde{X}^{-k\tau}_{-k\tau+jh}
        -
        \tilde{Y}^{-k\tau}_{-k\tau+jh},
        \quad
        \Delta \tilde{\Phi}^{X,Y}_j
        :=
        \Phi
        \big(
        \tilde{X}^{-k\tau}_{-k\tau+jh}
        \big)
        -
        \Phi
        \big(
        \tilde{Y}^{-k\tau}_{-k\tau+jh}
        \big), \\
        &
        \Delta \tilde{f}^{X,Y}_j
        :=
        f
        \Big(
        jh,
        \Phi
        \big(
        \tilde{X}^{-k\tau}_{-k\tau+jh}
        \big)
        \Big)
        -
        f
        \Big(
        jh,
        \Phi
        \big(
        \tilde{Y}^{-k\tau}_{-k\tau+jh}
        \big)
        \Big),\\
        &
        \Delta \tilde{g}^{X,Y}_j
        :=
        g
        \Big(
        jh,
        \Phi
        \big(
        \tilde{X}^{-k\tau}_{-k\tau+jh}
        \big)
        \Big)
        -
        g
        \Big(
        jh,
        \Phi
        \big(
        \tilde{Y}^{-k\tau}_{-k\tau+jh}
        \big)
        \Big),\\
        &
        \Delta \tilde{\mathcal{L}}^{X,Y}_{j,_{r_1,r_2}}
        :=
        \mathcal{L}^{r_1}
        g_{r_2}
        \Big(
        jh,
        \Phi
        \big(
        \tilde{X}^{-k\tau}_{-k\tau+jh}
        \big)
        \Big)
        -
        \mathcal{L}^{r_1}
        g_{r_2}
        \Big(
        jh,
        \Phi
        \big(
        \tilde{Y}^{-k\tau}_{-k\tau+jh}
        \big)
        \Big).\\
        \end{split}
    \end{equation}
    It is apparent to show that
    \begin{equation}
        \tilde{Z}_{j+1}
        =
        \Delta \tilde{\Phi}^{X,Y}_j
        +
        Ah
        \Delta \tilde{\Phi}^{X,Y}_j
        +
        h
        \Delta \tilde{f}^{X,Y}_j
        +
        \Delta \tilde{g}^{X,Y}_j
        \Delta W_{-k\tau+jh}
        +
        \sum_{r_1,r_2=1}^{m}
        \Delta \tilde{\mathcal{L}}^{X,Y}_{j,_{r_1,r_2}}
        \Pi^{-k\tau+jh,-k\tau+(j+1)h}_{r_1,r_2}.
    \end{equation}
    Taking expectations and 
    taking square on both sides,
    one can arrive at
    \begin{equation}
        \begin{split}
        \E
        \big[
        \|
        \tilde{Z}_{j+1}
        \|^2
        \big]
        & =
        \E
        \Big[
        \big\|
        \Delta \tilde{\Phi}^{X,Y}_j
        \big\|^2
        \Big]
        +
        h^2
        \E
        \Big[
        \big\|
        A
        \Delta \tilde{\Phi}^{X,Y}_j
        \big\|^2
        \Big]
        +
        h^2
        \E
        \Big[
        \big\|
        \Delta \tilde{f}^{X,Y}_j
        \big\|^2
        \Big]
        +
        \E
        \Big[
        \big\|
        \Delta \tilde{g}^{X,Y}_j
        \Delta W_{-k\tau+jh}
        \big\|^2
        \Big] \\
        &\quad +
        \E
        \bigg[
        \big\|
        \sum_{r_1,r_2=1}^{m}
        \Delta \tilde{\mathcal{L}}^{X,Y}_{j,_{r_1,r_2}}
        \Pi^{-k\tau+jh,-k\tau+(j+1)h}_{r_1,r_2}
        \big\|^2
        \bigg]
        +
        2h
        \E
        \Big[
        \big\langle
        \Delta \tilde{\Phi}^{X,Y}_j,
        A\Delta \tilde{\Phi}^{X,Y}_j
        \big\rangle
        \Big] \\
        & \quad +
        2h
        \E
        \Big[
        \big\langle
        \Delta \tilde{\Phi}^{X,Y}_j,
        \Delta \tilde{f}^{X,Y}_j
        \big\rangle
        \Big] 
        +
        2h^2
        \E
        \Big[
        \big\langle
        A\Delta \tilde{\Phi}^{X,Y}_j,
        \Delta \tilde{f}^{X,Y}_j
        \big\rangle
        \Big] .
        \end{split}
    \end{equation}\
    Noting that
    \begin{align}
        \E
        \Big[
        \big\|
        \Delta \tilde{g}^{X,Y}_j
        \Delta W_{-k\tau+jh}
        \big\|^2
        \Big]
        &=
        h
        \E
        \Big[
        \big\|
        \Delta \tilde{g}^{X,Y}_j
        \big\|^2
        \Big], \\
         \E
        \Big[
        \big\|
        \sum_{r_1,r_2=1}^{m}
        \Delta \tilde{\mathcal{L}}
        ^{X,Y}_{j,_{r_1,r_2}}
        \Pi
        ^{-k\tau+jh,-k\tau+(j+1)h}
        _{r_1,r_2}
        \big\|^2
        \Big]
        &
        \leq
        h^2
        \sum_{r_1,r_2=1}^{m}
        \E
        \Big[
        \big\|
        \Delta \tilde{\mathcal{L}}^{X,Y}_{j,_{r_1,r_2}}
        \big\|^2
        \Big].
    \end{align} 
    In view of   \eqref{eq_PMM:the_esti_h^2E|Delta_L|^2}, one can see that
    \begin{equation}
        h^2
        \sum_{r_1,r_2=1}^{m}
        \E
        \Big[
        \big\|
        \Delta \tilde{\mathcal{L}}^{X,Y}_{j,_{r_1,r_2}}
        \big\|^2
        \Big] 
        \leq
        \beta_{\mathcal{L}}^2
        h
        \E
        [
        \|
        \Delta \tilde{\Phi}^{X,Y}_j
        \|^2
        ] 
        \leq
        2 \beta_{\mathcal{L}}^2
        h
        \E
        [
        \|
        \Delta \tilde{\Phi}^{X,Y}_j
        \|^2
        ].
    \end{equation}
Using the Cauchy-Schwarz inequality
    leads to
    \begin{equation}
        2h^2
        \E
        \Big[
        \big\langle
        A\Delta \tilde{\Phi}^{X,Y}_j,
        \Delta \tilde{f}^{X,Y}_j
        \big\rangle
        \Big]
        \leq
        2h^2
        \E
        \Big[
        \big\|
        A\Delta \tilde{\Phi}^{X,Y}_j
        \big\|
        \cdot
        \big\|
        \Delta \tilde{f}^{X,Y}_j
        \big\|
        \Big].
    \end{equation}
    Recalling Assumption \ref{ass_PMM}, Assumption \ref{ass_PMM:generalizedz_monotonicity_condition}
    and Lemma \ref{lem_PMM:Phi(x)},
    one can deduce
    \begin{equation}
        \begin{split}
        \E
        \big[
        \|
        \tilde{Z}_{j+1}
        \|^2
        \big]
        & \leq
        \E
        \Big[
        \big\|
        \Delta \tilde{\Phi}^{X,Y}_j
        \big\|^2
        \Big]
        +
        2h
        \bigg\{
        \E
        \Big[
        \big\langle
        \Delta \tilde{\Phi}^{X,Y}_j,
        A\Delta \tilde{\Phi}^{X,Y}_j
        \big\rangle
        \Big]
        +
        \E
        \Big[
        \big\langle
        \Delta \tilde{\Phi}^{X,Y}_j,
        \Delta \tilde{f}^{X,Y}_j
        \big\rangle
        \Big] \\
        &\quad +
        \tfrac{2q-1}{2}
        \E
        \Big[
        \big\|
        \Delta \tilde{g}^{X,Y}_j
        \big\|^2
        \Big]
        +
        \beta
        _{\mathcal{L}}^2
        \sum_{r_1,r_2=1}^{m}
        \E
        \Big[
        \big\|
        \Delta \tilde{\mathcal{L}}^{X,Y}_{j,_{r_1,r_2}}
        \big\|^2
        \Big]
        \bigg\} 
        \\
        &\quad +
        \lambda_d^2h^2
        \E
        \Big[
        \big\|
        \Delta \tilde{\Phi}^{X,Y}_j
        \big\|^2
        \Big]
        +
        2\lambda_d\beta_f
        h^{1+\frac{\gamma+1}{2\gamma}}
        \E
        \Big[
        \big\|
        \Delta \tilde{\Phi}^{X,Y}_j
        \big\|^2
        \Big]
        +
        \beta_f^2
        h^{1+\frac{1}{\gamma}}
        \E
        \Big[
        \big\|
        \Delta \tilde{\Phi}^{X,Y}_j
        \big\|^2
        \Big] \\
        & \leq
        \Big(
        1
        -
        2
        (
        \lambda_1
        -K_2
        -
        \beta_{\mathcal{L}}^2
        )h
        \Big)
        \E
        \big[
        \|
        \tilde{Z}_j
        \|^2
        \big]
        +
        \Big(
        \lambda_d^2h
        +
        2\lambda_d\beta_f
        h^{\frac{\gamma+1}{2\gamma}}
        +
        \beta_f^2
        h^{\frac{1}{\gamma}}
        \Big) h
        \E
        \big[
        \|
        \tilde{Z}_j
        \|^2
        \big].
        \end{split}
    \end{equation}
According to 
    $\gamma \geq 1 $,
    one can get
    \begin{equation}
        \lambda_d^2h
        +
        2\lambda_d\beta_f
        h^{\frac{\gamma+1}{2\gamma}}
        +
        \beta_f^2
        h^{\frac{1}{\gamma}}
        \leq
        (\lambda_d+\beta_f)^2
        h^{\frac{1}{\gamma}}.
    \end{equation}
    Here we select an appropriate $h$ such that
    such that
    \begin{equation}
        (\lambda_d+\beta_f)^2
        h^{\frac{1}{\gamma}}
        \leq
        \lambda_1
        -K_2
        -
        \beta_{\mathcal{L}}^2,
    \end{equation}
    which leads to 
    \begin{equation}
        h
        \in
    \Big(
    0,
    \min
    \Big\{
    \tfrac
    {
    (\lambda_1
    -K_2
    -
\beta_{\mathcal{L}}^2)^{\gamma}}{(\lambda_d+\beta_f)^{2\gamma}},
    \tfrac{1}{\lambda_1-K_2-
\beta_{\mathcal{L}}^2}
    ,
    1
    \Big\}
    \Big),
    \end{equation}
    \begin{equation}
        \E \big[
        \|
        \tilde{Z}_{j+1}
        \|
        \big]
        \leq
        \big(
        1
        -
        (
        \lambda_1
        -K_2
        -
        \beta_{\mathcal{L}}^2
        )h
        \big)
         \E \big[
        \|
        \tilde{Z}_{j}
        \|
        \big]
        \leq
        e
        ^
        {-
        (
        \lambda_1
        -K_2
        -
        \beta_{\mathcal{L}}^2
        )}
        (j+1)h
        \E
        \big[
        \| \xi-\eta \|^{2}
        \big].
    \end{equation}  
    The proof is completed.
\end{proof}

Within the framework of Theorem 3.4 as presented by \cite{feng2017numerical},
we can derive the existence and uniqueness of the random periodic solution to the projected Milstein method \eqref{eq:the_projected_milstein_method}.
\begin{thm}
\label{thm:PMM random period solution}
   Let Assumption \ref{ass_PMM}  hold. 
   For 
   $
         h
        \in
    \Big(
    0,
    \min
    \Big\{
    \tfrac
    {
    (
    \lambda_1
    -K_2
    -
    \beta_{\mathcal{L}}^2)
    ^{\gamma}}
    {(\lambda_d+\beta_f)^{2\gamma}},
    \tfrac{1}{\lambda_1-K_2-
\beta_{\mathcal{L}}^2}
    ,
    1
    \Big\}
    \Big)
    $,
   the projected Milstein method \eqref{eq:the_projected_milstein_method} admits a random period solution $\tilde{X}^{*}_{t} \in L^{2}(\Omega)$ such that
\begin{equation}
  \lim_{k \rightarrow 
  \infty }
   \E
   \Big[
   \big\|
   \tilde{X}_{-k\tau+jh}^{-k\tau}(\xi)
   -\tilde{X}^{*}_{t}
   \big\|^{2}
   \Big]
   =0.
\end{equation}
\end{thm}

\begin{cor}
\label{cor_PMM:ch}
    Let $X_t^{*}$ be the random periodic solution of SDE \eqref{eq_PMM:Problem_SDE}
    and let $\tilde{X}_t^{*}$
    be the numerical approximating random
    periodic solution of
    \eqref{eq:the_projected_milstein_method},
    with $\gamma \in \big[1,\frac{p^*}{5}\big]$.
    Suppose that Assumption \ref{ass_PMM} holds,
    then there exists $C>0$,
    independent of $h$,
    such that
    \begin{equation}
        \|
        X^*_t
        -\tilde{X}^*_t
        \|_{L^{2}(\Omega;\R^d)}
        \leq
        Ch.
    \end{equation}
\end{cor}

\begin{proof}[Proof of Corollary \ref{cor_PMM:ch}]
Noting that
\begin{equation}
\E[
\| X^{*}_{t}-\tilde{X}^{*}_{t}\|^{2}]
\leq
\limsup_{k}
\big[
\E[\| X^{*}_{t}-X^{-k\tau}_{t}\|^{2}]
+\E[\| X^{-k\tau}_{t}-\tilde{X}^{-k\tau}_{t}\|^{2}]
+\E[\|\tilde{X}^{-k\tau}_{t}-\tilde{X}^{*}_{t}\|^{2}]\big],
\end{equation}
the conclusion can be obtained by Theorem \ref{thm:unique_random_periodic_solution}, Theorem \ref{thm_PMM:error analysis},
Theorem \ref{thm:PMM random period solution}.
\end{proof}

\section{Numerical experiments}
\label{sec_PMM:numerical results}
In this section,
we conduct numerical experiments to demonstrate the previous theoretical results. 
Let us focus on the following
one-dimensional SDE
\begin{equation}
    \label{eq_PMM:example_1}
     \dd X^{t_{0}}_{t}
    =
    \big(
    -2\pi 
    X^{t_{0}}_{t}
     +
     X^{t_{0}}_{t}
     -
     (
     X^{t_{0}}_{t}
     )^3
     +\cos(\pi t)
     \big)
     \, \dd t
     +
     \big(
     1+
     (X^{t_{0}}_{t})^2
     +
     \cos(\pi t)
     \big)
      \, \dd W_{t},
      \:
      X_{t_{0}}^{t_{0}}  = \xi
      .
\end{equation}

Theorem \ref{thm:PMM random period solution} indicates that the projected Milstein method applied to \eqref{eq_PMM:example_1} admits a unique random periodic solution.
Let us first numerically verity that, the random periodic solution of projected Milstein method does not depend on the initial values.
To fulfill this, we choose the time grad between $t_0=-20$
and $T=0$ with stepsize 0.01,
and select two initial values to be $\xi = 0.8$ and $\xi = -0.5$.
As presented in Figure \ref{fig1},
two paths match after a very short time.

\begin{figure}[H]
   \centering
\includegraphics[scale=0.5]{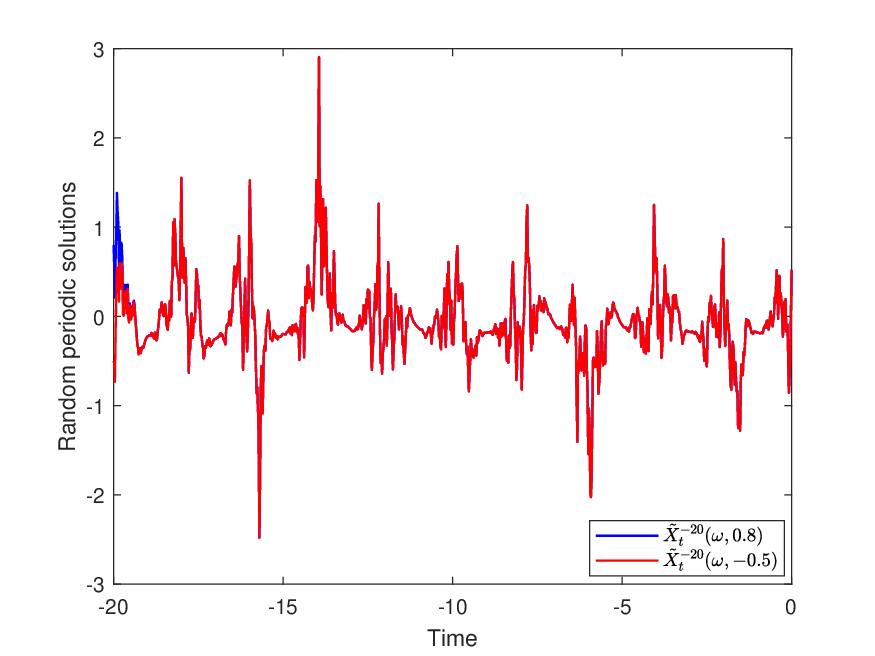}
	\caption[Figure 1.]
	{Two paths generated by projected Milstein methods from differential initial conditions.}
 \label{fig1}
\end{figure}

To show the periodicity of generated solutions obtained by
projected Milstein method,
we first simulate two paths
$\tilde{X}^*_{t}(\omega)=\tilde{X}^{-20}_{t}(\omega,0.3)$
for $t \in [2,6]$ and
$\tilde{X}^*_{t}(\theta_{-2}\omega)=\tilde{X}^{-20}_{t}(\theta_{-2}\omega,0.3)$
for $t \in [4,8]$,
with the same $\omega$ and initial value 0.3.
Figure \ref{fig2} shows that the two segmented processes resemble
each other with the periodic $\tau=2$.

\begin{figure}[H]
   \centering
\includegraphics[scale=0.5]{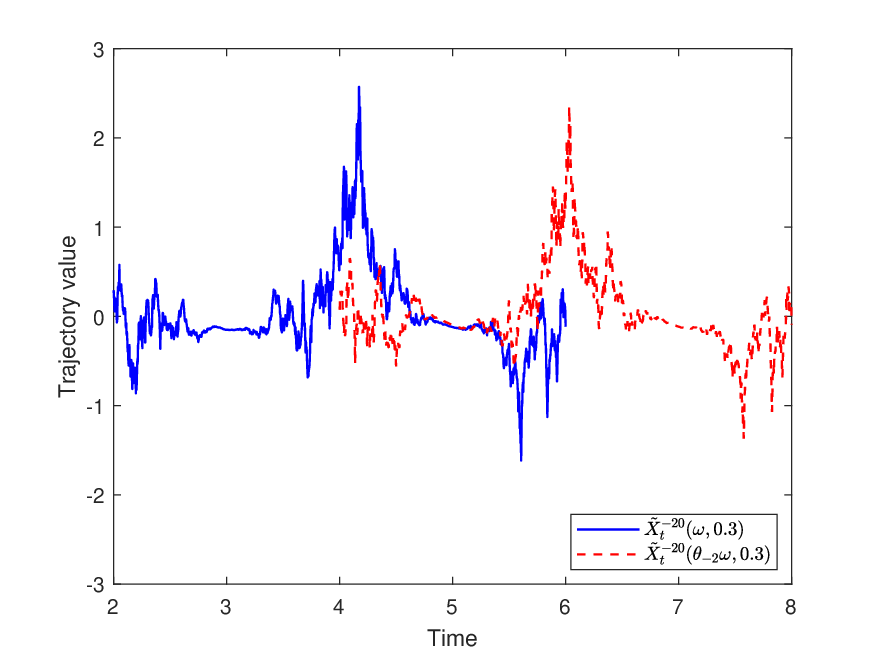}
\caption[Figure 2.]
	{Simulations of the process 
 ${\tilde{X}_{t}^{-20}(\omega,0.3),2 \leq t \leq 6}$
 and
 ${\tilde{X}_{t}^{-20}(\theta_{-2}\omega,0.3),4 \leq t \leq 8}$. }
 \label{fig2}
\end{figure}

To test the mean-square convergence rates,
we depict in Figure \ref{fig3} 
mean-square approximation errors $e_h$ against against five different stepsizes $h=2^{-i} \times 20, i=8,9,...,12$ on a log-log scale.
Also, two reference lines of slope $1$ and $\frac{1}{2}$ are given there.
In Figure \ref{fig3},
the mean-square convergence rates of two methods are depicted on a log-log scale.
There one can easily see that the mean-square convergence rate of PEM is beyond than $0.5$,
as opposed to a convergence rate close to $1$ for PMM.

\begin{figure}[H]
	\centering	\includegraphics[scale=0.5]{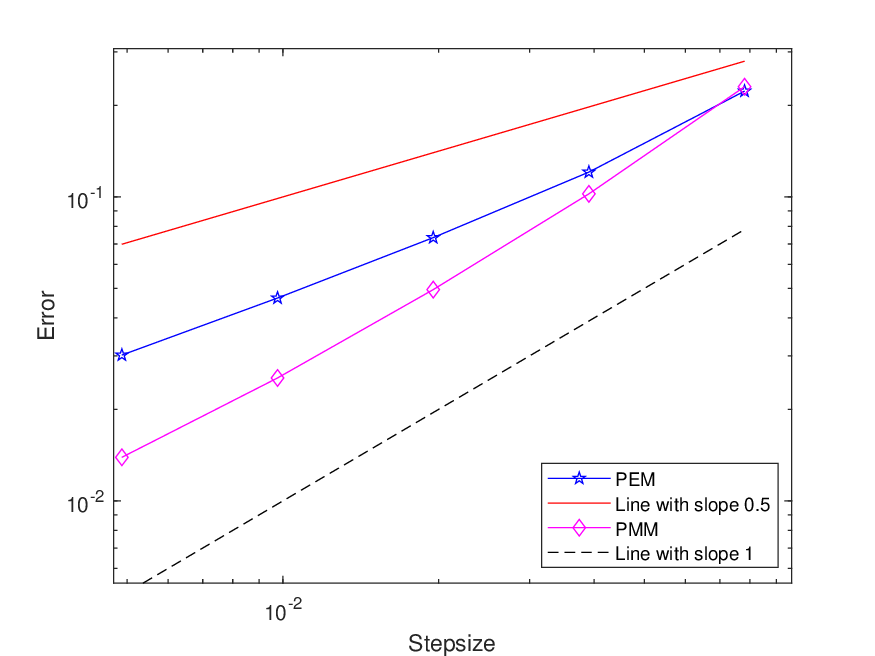}
 \caption[Figure 3.]
{The mean-square error plot  of \eqref{eq_PMM:example_1}.}
\label{fig3}
\end{figure}

\vskip6mm
\bibliographystyle{plain}

\begin{thebibliography}{10}

\bibitem{andersson2017mean}
Adam Andersson and Raphael Kruse.
\newblock Mean-square convergence of the {BDF2-Maruyama} and backward {Euler} schemes for {SDE} satisfying a global monotonicity condition.
\newblock {\em BIT Numerical Mathematics}, 57(1):21--53, 2017.

\bibitem{beyn2016stochastic}
Wolf-J{\"u}rgen Beyn, Elena Isaak, and Raphael Kruse.
\newblock {Stochastic C-stability and B-consistency of explicit and implicit Euler-type schemes}.
\newblock {\em Journal of Scientific Computing}, 67:955--987, 2016.

\bibitem{beyn2017stochastic}
Wolf-J{\"u}rgen Beyn, Elena Isaak, and Raphael Kruse.
\newblock {Stochastic C-stability and B-consistency of explicit and implicit Milstein-type schemes}.
\newblock {\em Journal of Scientific Computing}, 70:1042--1077, 2017.

\bibitem{chen2024stochastic}
Ziheng Chen, Liangmin Cao, and Lin Chen.
\newblock Stochastic theta methods for random periodic solution of stochastic differential equations under non-globally lipschitz conditions.
\newblock {\em Numerical Algorithms}, pages 1--31, 2024.

\bibitem{feng2017numerical}
Chunrong Feng, Yu~Liu, and Huaizhong Zhao.
\newblock Numerical approximation of random periodic solutions of stochastic differential equations.
\newblock {\em Zeitschrift f{\"u}r angewandte Mathematik und Physik}, 68(5):119, 2017.

\bibitem{feng2016anticipating}
Chunrong Feng, Yue Wu, and Huaizhong Zhao.
\newblock Anticipating random periodic solutions—i. {SDEs} with multiplicative linear noise.
\newblock {\em Journal of Functional Analysis}, 271(2):365--417, 2016.

\bibitem{feng2012random}
Chunrong Feng and Huaizhong Zhao.
\newblock Random periodic solutions of {SPDEs} via integral equations and wiener--sobolev compact embedding.
\newblock {\em Journal of Functional Analysis}, 262(10):4377--4422, 2012.

\bibitem{feng2020random}
Chunrong Feng and Huaizhong Zhao.
\newblock Random periodic processes, periodic measures and ergodicity.
\newblock {\em Journal of Differential Equations}, 269(9):7382--7428, 2020.

\bibitem{feng2011pathwise}
Chunrong Feng, Huaizhong Zhao, and Bo~Zhou.
\newblock Pathwise random periodic solutions of stochastic differential equations.
\newblock {\em Journal of Differential Equations}, 251(1):119--149, 2011.

\bibitem{guo2023order}
Yujia Guo, Xiaojie Wang, and Yue Wu.
\newblock Order-one convergence of the {Backward Euler} method for random periodic solutions of semilinear {SDEs}.
\newblock {\em arXiv preprint arXiv:2306.06689}, 2023.

\bibitem{guo2024projected}
Yujia Guo, Xiaojie Wang, and Yue Wu.
\newblock A projected {Euler} method for random periodic solutions of semi-linear {SDEs} with non-globally lipschitz coefficients.
\newblock {\em arXiv preprint arXiv:2406.16089}, 2024.

\bibitem{higham2002strong}
Desmond~J Higham, Xuerong Mao, and Andrew~M Stuart.
\newblock Strong convergence of {Euler}-type methods for nonlinear stochastic differential equations.
\newblock {\em SIAM journal on numerical analysis}, 40(3):1041--1063, 2002.

\bibitem{hutzenthaler2015numerical}
Martin Hutzenthaler and Arnulf Jentzen.
\newblock {\em {Numerical approximations of stochastic differential equations with non-globally Lipschitz continuous coefficients}}, volume 236.
\newblock American Mathematical Society, 2015.

\bibitem{hutzenthaler2020perturbation}
Martin Hutzenthaler and Arnulf Jentzen.
\newblock On a perturbation theory and on strong convergence rates for stochastic ordinary and partial differential equations with nonglobally monotone coefficients.
\newblock {\em The Annals of Probability}, 48(1):53--93, 2020.

\bibitem{hutzenthaler2011strong}
Martin Hutzenthaler, Arnulf Jentzen, and Peter~E Kloeden.
\newblock {Strong and weak divergence in finite time of Euler's method for stochastic differential equations with non-globally Lipschitz continuous coefficients}.
\newblock {\em Proceedings of the Royal Society A: Mathematical, Physical and Engineering Sciences}, 467(2130):1563--1576, 2011.

\bibitem{hutzenthaler2012strong}
Martin Hutzenthaler, Arnulf Jentzen, and Peter~E Kloeden.
\newblock Strong convergence of an explicit numerical method for {SDEs} with nonglobally {Lipschitz} continuous coefficients.
\newblock {\em The Annals of Applied Probability}, 22(4):1611--1641, 2012.

\bibitem{kelly2018adaptive}
C.~Kelly and G.~J. Lord.
\newblock Adaptive time-stepping strategies for nonlinear stochastic systems.
\newblock {\em IMA Journal of Numerical Analysis}, 38(3):1523--1549, 2018.

\bibitem{kloeden1992stochastic}
Peter~E Kloeden and Eckhard Platen.
\newblock {\em Numerical solution of stochastic differential equations}.
\newblock Springer, 1992.

\bibitem{li2020explicit}
Xiaoyue Li and George Yin.
\newblock Explicit {Milstein} schemes with truncation for nonlinear stochastic differential equations: Convergence and its rate.
\newblock {\em Journal of Computational and Applied Mathematics}, 374:112771, 2020.

\bibitem{mao2015truncated}
Xuerong Mao.
\newblock The truncated {Euler Maruyama} method for stochastic differential equations.
\newblock {\em Journal of Computational and Applied Mathematics}, 290:370--384, 2015.

\bibitem{milstein2004stochastic}
Grigori~N Milstein and Michael~V Tretyakov.
\newblock {\em Stochastic numerics for mathematical physics}, volume~39.
\newblock Springer, 2004.

\bibitem{moradi2024random}
Afsaneh Moradi and Raffaele D’Ambrosio.
\newblock Random periodic solutions of {SDEs}: Existence, uniqueness and numerical issues.
\newblock {\em Communications in Nonlinear Science and Numerical Simulation}, 128:107586, 2024.

\bibitem{pang2023projected}
Chenxu Pang, Xiaojie Wang, and Yue Wu.
\newblock Projected langevin monte carlo algorithms in non-convex and super-linear setting.
\newblock {\em arXiv preprint arXiv:2312.17077}, 2023.

\bibitem{pang2024linear}
Chenxu Pang, Xiaojie Wang, and Yue Wu.
\newblock Linear implicit approximations of invariant measures of semi-linear {SDEs} with non-globally lipschitz coefficients.
\newblock {\em Journal of Complexity}, page 101842, 2024.

\bibitem{sabanis2016euler}
Sotirios Sabanis.
\newblock Euler approximations with varying coefficients: the case of superlinearly growing diffusion coefficients.
\newblock {\em The Annals of Applied Probability}, 26(4):2083--2105, 2016.

\bibitem{wang2023mean}
Xiaojie Wang.
\newblock Mean-square convergence rates of implicit {Milstein} type methods for {SDEs} with non-lipschitz coefficients.
\newblock {\em Advances in Computational Mathematics}, 49(3):37, 2023.

\bibitem{wang2013tamed}
Xiaojie Wang and Siqing Gan.
\newblock The tamed {Milstein} method for commutative stochastic differential equations with non-globally lipschitz continuous coefficients.
\newblock {\em Journal of Difference Equations and Applications}, 19(3):466--490, 2013.

\bibitem{wang2020mean}
Xiaojie Wang, Jiayi Wu, and Bozhang Dong.
\newblock Mean-square convergence rates of stochastic theta methods for {SDEs} under a coupled monotonicity condition.
\newblock {\em BIT Numerical Mathematics}, 60(3):759--790, 2020.

\bibitem{wu2023backward}
Yue Wu.
\newblock Backward {Euler--Maruyama} method for the random periodic solution of a stochastic differential equation with a monotone drift.
\newblock {\em Journal of Theoretical Probability}, 36(1):605--622, 2023.

\bibitem{zhao2009random}
Huaizhong Zhao and Zuo-Huan Zheng.
\newblock Random periodic solutions of random dynamical systems.
\newblock {\em Journal of Differential equations}, 246(5):2020--2038, 2009.

\end{thebibliography}

\end{document}